\documentclass[reqno,11pt]{article}
\usepackage{graphicx, amsmath, amsthm,amssymb}  
\usepackage{tikz}
\usetikzlibrary{matrix,arrows,calc}
\usepgflibrary{fpu}
\usepackage{MnSymbol} 
\usepackage{gensymb}

\input{allencahn.sty}

\RequirePackage{color}

\definecolor{mygreen}{rgb}{0.1,0.5,0.1}

\begin{document}


\title{An Eyring--Kramers law for the stochastic \\ Allen--Cahn equation in
dimension two}
\author{Nils Berglund, Giacomo Di Ges\`u and Hendrik Weber}
\date{January 20, 2017}   

\maketitle

\begin{abstract}
We study spectral Galerkin approximations of an Allen--Cahn equation over the
two-dimensional torus perturbed by weak space-time white noise of strength
$\sqrt{\eps}$. We introduce a Wick renormalisation of the equation in order to
have a system that is well-defined as the regularisation is removed. We show
sharp upper and lower bounds on the transition times from a neighbourhood of the
stable configuration $-1$ to the stable configuration $1$ in the asymptotic
regime $\eps \to 0$. These estimates are uniform in the discretisation parameter
$N$, suggesting an Eyring--Kramers formula for the limiting renormalised
stochastic PDE. The effect of the \lq\lq infinite renormalisation\rq\rq\ is to
modify the prefactor and to replace the ratio of determinants in the
finite-dimensional Eyring--Kramers law by a renormalised Carleman--Fredholm
determinant.
\end{abstract}

\leftline{\small 2010 {\it Mathematical Subject Classification.\/} 
60H15, 
35K57 (primary),    
81S20,   
82C28 (secondary).  
}


\section{Introduction}
\label{sec_intro}

Metastability is a common physical phenomenon in which a system spends a long
time in metastable states before  reaching  its equilibrium.  One of the most
classical mathematical models where this phenomenon has been studied rigorously
is the overdamped motion of a particle in a potential $V$, given by the It\^o
stochastic differential equation
\begin{equation}\label{e:SDE1}
\6 x(t) = -\nabla V(x(t))\6t + \sqrt{2 \eps} \6w(t)\;.
\end{equation}
For small noise strength $\eps$ solutions of \eqref{e:SDE1} typically spend long
stretches of time near local minima of the potential $V$ with occasional,
relatively quick  transitions between these local minima. The mean transition
time between  minima is then governed by the Eyring--Kramers law
\cite{eyring1935,Kramers}: If $\tau$ denotes the expected hitting time of a
neighbourhood of a local minimum $y$ of the solution of \eqref{e:SDE1} started
in another minimum $x$, and under suitable assumptions on the potential $V$, the
Eyring--Kramers law  gives the asymptotic expression
\begin{equation}\label{e:KE1}
\E[ \tau] = \frac{2 \pi}{|\lambda_0 (z)|}  \sqrt{\frac{|\det D^2 V(z)| }{ \det
D^2 V(x)}} \e^{[V(z) -V(x)]/\eps} [1+ \orderone{\eps}]\;,
\end{equation}
where $z$ is the \emph{relevant saddle} between $x$ and $y$, and $\lambda_0(z)$
is the (by assumption) unique negative eigenvalue of the Hessian $D^2 V(z)$
(more precise bounds on the error term $\orderone{\eps}$ are available). The
right exponential rate in this formula was established rigorously
using large deviation theory \cite{FreidlinWentzell}. Rigorous proofs including
the prefactor were given in
\cite{sugiura1995metastable,BEGK,HelfferKleinNier04}, see
\cite{berglund2011kramers} for a recent survey.

It is natural to study metastability in high- or infinite-dimensional systems
and to seek an extension of the Eyring--Kramers law. In this direction the
Allen--Cahn equation perturbed by a small noise term is a natural model to
study. It is given by the stochastic PDE
\begin{equation}\label{e:ACE}
\partial_t \phi(t,x) = \Delta \phi(t,x) - (\phi(t,x)^3 - \phi(t,x)) + \sqrt{2
\eps} \xi(t,x)\;,
\end{equation}
where $\xi$ is a noise term to be described below. Just like in~\eqref{e:SDE1}
the deterministic part of this equation (set $\xi = 0$ in \eqref{e:ACE}) has
gradient form and the relevant potential is given by
\begin{equation}
\label{eq:potential} 
V(\phi) = \int  
\Bigl(\frac12 |\nabla \phi|^2 - \frac12 \phi^2 +\frac14 \phi^4 \Bigr) \6x\;.
\end{equation}
The natural  choice of noise term $\xi$ is (at least formally) given by
space-time white noise because  this choice is compatible with the scalar
product used to define the deterministic gradient flow and  it makes the
dynamics given by \eqref{e:ACE} reversible (in the sense that the system
satisfies detailed balance when in its statistical equilibrium state). The
constant profiles $\phi_{\pm}(x) =  \pm 1$ are stable solutions of the
deterministic system and it is natural to ask how long a small noise term
typically needs to move the system from one of these stable profiles to the
other one. 

In the case where equation~\eqref{e:ACE} is solved over a $1+1$-dimensional
time-space domain $(t,x) \in [0,\infty) \times [0,L]$  this question was 
first studied in~\cite{FarisJona} on the level of large deviations,
yielding the correct exponent in~\eqref{e:KE1}. The problem of obtaining sharper
asymptotics with the correct prefactor was first considered
in~\cite{BarretBovierMeleard}, and infinite-dimensional versions of the
Eyring--Kramers formula were established in~\cite{Barret2,berglund2013sharp}.
Let $\tau$ denote the first-hitting time of a solution of \eqref{e:ACE}
starting near the constant profile $\phi_-$ of a neighbourhood of the constant
profile $\phi_+=1$. In \cite{berglund2013sharp} it was shown, for example in the
case where \eqref{e:ACE} is endowed with periodic boundary conditions on $[0,L]$
and $L<2\pi$, that 
\begin{equation}\label{e:KE2}
\E [\tau ]=  \frac {2 \pi}{ |\lambda_0|}\prod_{k \in \Z}
\sqrt{\frac{|\lambda_k|}{\nu_k}}  \e^{[V(\phi_0)-V(\phi_-)]/\eps} [1+
\orderone{\eps}]\;,
\end{equation}
where $k$ plays the role of a wave number in a Fourier decomposition.
The purpose of the condition $L < 2\pi$ is to ensure that the constant profile
$\phi_0 = 0$ is the relevant saddle between the two stable minima $\phi_{\pm}$;
but situations for longer intervals and different choices of boundary conditions
are also described in \cite{berglund2013sharp}. The sequences $\lambda_k$,
$\nu_k$ appearing in this expression are the eigenvalues of the Hessian of $V$
around $\phi_0$ and $\phi_-$, the operators $ - \partial_x^2 -1$ and
$-\partial_x^2 +2$, both endowed with periodic boundary condition (the
corresponding eigenfunctions being simply Fourier modes). All of these
eigenvalues are strictly positive, except for $\lambda_0 = -1$. Leaving out the
factor $k=0$, the infinite product in \eqref{e:KE2} can be written as
\begin{equation}\label{det1}
\prod_{k \neq 0} \sqrt{\frac{\lambda_k}{\nu_k}} 
= \prod_{k \neq 0} \sqrt{\biggl(1 + \frac{\nu_k -
\lambda_k}{\lambda_k}\biggr)^{-1}} 
= \frac{1}{\sqrt{\det(\mathrm{Id} + 3P_\perp(-\partial_x^2 {}-{}
1)^{-1})}}\;,
\end{equation}
where $P_\perp$ projects on the complement of the $k=0$ mode and the
operator $(-\partial_x^2 {}-{} 1)$ acts on zero-mean functions. This expression
converges, because $P_\perp(-\partial_x^2 -  1)^{-1}$ is a trace-class
operator, so that the infinite-dimensional (Fredholm) determinant is
well-defined (see for instance~\cite{Forman1987,Maier_Stein_PRL_01}).

When trying to extend this result to higher spatial dimensions two problems
immediately present themselves. First, for spatial dimension $d \geqs 2$   the
Allen--Cahn equation as stated in \eqref{e:ACE} fails to be well-posed: in this
situation already the linear stochastic heat equation (drop the non-linear term
$-(\phi^3 -\phi)$ in \eqref{e:ACE}) has distribution-valued solutions due to
the irregularity of the white noise $\xi$.
In this regularity class $-(\phi^3- \phi)$ does not have a canonical definition.
As an illustration for the problems caused by this irregularity, it was shown
in~\cite{HairerRyserWeber}
that for fixed noise strength $\eps$ finite-dimensional spectral Galerkin 
approximations\footnote{In fact, in \cite{HairerRyserWeber,daPratoDebussche} the
nonlinearity $\phi^3$ is not projected onto a finite dimensional space, but this
does not affect the result.}
\begin{equation}\label{e:ACE-approx}
\partial_t \phi_N = \Delta \phi_N -  (P_N \phi_N^3 - \phi_N) + \sqrt{2 \eps}
\xi_N
\end{equation} 
defined over a two-dimensional torus converge to a trivial limit as the
approximation parameter $N$ goes to $\infty$ (precise definitions of the
finite-dimensional noise $\xi_N$ and the projection operator $P_N$ are given in
Section \ref{sec_results} below). A second related problem is, that for $d \geqs
2$ the infinite product appearing in \eqref{e:KE2} converges to $0$,
corresponding to the fact that for $d \geqs 2$  the operator $3P_\perp(-\Delta
{}-{} 1)^{-1}$ fails to be trace-class so that the Fredholm determinant
$\det(\mathrm{Id} {}+{} 3P_\perp(-\Delta {}-{} 1)^{-1})$ is not
well-defined.

On the level of the $N \to \infty$ limit for fixed $\eps$ the idea of
renormalisation, inspired by ideas from Quantum Field Theory (see e.g.
\cite{Glimm_Jaffe_81}), has been very successful over the last years. Indeed, in
\cite{daPratoDebussche} it was shown that in the two-dimensional case, if the
approximations in \eqref{e:ACE-approx} are replaced by 
\begin{equation}\label{e:ACE2}
\partial_t \phi_N = \Delta \phi_N - (P_N \phi_N^3 - 3\eps C_N \phi_N-  \phi_N) +
\sqrt{2 \eps} \xi_N\;,
\end{equation}
for a particular choice of logarithmically divergent constants $C_N$ (see
\eqref{eq:cN}  below), the solutions do converge to a non-trivial limit which
can be interpreted as \emph{renormalised solutions} of \eqref{e:ACE}. This
result (for a different choice of $C_N$) was spectacularly extended to three
dimensions in Hairer's pioneering work on regularity structures \cite{Hairer}.
For spatial dimension $d \geqs 4$, equation \eqref{e:ACE} fails to satisfy a
\emph{subcriticality condition} (see \cite{Hairer}) and non-trivial renormalised
solutions are not expected to exist. 

Note that formally the extra term $3\eps C_N\phi_N$ moves the stable
solutions further apart to $\pm \sqrt{3\eps C_N +1}$ (and ultimately to $\pm
\infty$ as $N \to \infty$). Note furthermore that while the constants $C_N$
diverge as $N$ goes to $\infty$, for fixed $N$ they are multiplied with a small
factor $\eps$. This suggests that in the small-noise regime the renormalised
solutions may still behave as perturbations of the Allen--Cahn equation, despite
the presence of the infinite renormalisation constant. In \cite{HairerWeber}
this intuition was confirmed on the level of large deviations. There it was
shown that both in two and three dimensions the renormalised stochastic PDE 
satisfies a large-deviation principle as $\eps \to 0$, with respect to a
suitable topology and with rate functional given by 
\begin{equation}
\label{eq:LDP} 
\mathcal{I}(\phi) = \int_0^T \!\!\!\int \big( \partial_t \phi - \big( \Delta
\phi - (\phi^3 - \phi) \big) \big)^2 \6x\6t\;,
\end{equation}
which is exactly the \lq\lq correct\rq\rq\ rate functional one would obtain by
formally applying Freidlin--Wentzell theory to \eqref{e:ACE} without any regard
to renormalisation. 
Results in a similar spirit had previously been obtained in
\cite{jona1990large,aida2009semi}.

The purpose of this article is to show that the renormalised solutions have the
right asymptotic small-noise behaviour even beyond large deviations,  and to
establish an Eyring--Kramers formula in this framework. As remarked also
in~\cite{Rolland_Bouchet_Simonnet_16} nothing seems to be known at this level so
far. The key observation is that the introduction of the infinite constant not
only permits  to define the dynamics, but that it also fixes the problem of
vanishing prefactor in the Eyring--Kramers law \eqref{e:KE2}. More
precisely, we argue that in two dimensions the correct Eyring--Kramers formula
for the renormalised SPDE is
\begin{equation}\label{e:KE3}
\E[ \tau] = \frac{2 \pi}{|\lambda_0|} \sqrt{\prod_{k \in \Z^2} 
\frac{\abs{\lambda_k}}{\nu_k}
\,\exp\biggset{\frac{\nu_k - \lambda_k}{\abs{\lambda_k}}}} 
\e^{[V(\phi_0)-V(\phi_-)]/\eps} [1+ \orderone{\eps}]\;,
\end{equation}
where as above the $\lambda_k$ and $\nu_k$   are the eigenvalues of $-\Delta -1$
and $-\Delta +2$, now  indexed by a vectorial wave number $k \in \Z^2$. In
functional-analytic terms this means that due to the presence of the infinite
renormalisation constant the regular determinant from \eqref{e:KE2} is replaced
by a renormalised or Carleman--Fredholm determinant of the operator $\mathrm{Id}
+ 3P_\perp(-\Delta-1)^{-1}$. Unlike the \lq\lq usual\rq\rq\ determinant,  the
Carleman--Fredholm determinant is defined for the  class of Hilbert--Schmidt
perturbations of the identity and not only for the smaller class of trace-class
perturbations of the identity. Recall, that $(-\Delta -1)^{-1}$ is
Hilbert--Schmidt both in two and three dimensions, but not for $d \geqs 4$. It
is striking to note that these are exactly the dimensions in which a
renormalised solution to the Allen--Cahn (or $\Phi^4$) equation can be
constructed. 

In order to illustrate our result in the easiest possible situation we only
consider the case of the Allen--Cahn equation in a small domain $\T^2 = [0,L]^2$
of size $L < 2 \pi$ with periodic boundary conditions. As in the one-dimensional
case this assumption guarantees that the constant profile $\phi_0$ is the
relevant saddle. We make use of the $\pm 1$ symmetry of the system to simplify
some arguments. Throughout the article, we work in the framework of the
finite-dimensional spectral Galerkin approximation \eqref{e:ACE2} and
derive  asymptotic
bounds for the expected transition time as $\eps \to 0$ which are uniform in the
approximation parameter $N \to \infty$. 

On the technical level, our analysis builds on the potential-theoretic approach
developed in \cite{BEGK}, which allows to express expected transition times
in terms of committor functions and capacities, that can be estimated using a
variational principle. As we work in finite dimensions throughout, we can avoid
making any use of the analytic tools developed in recent years to deal with
singular SPDEs. A crucial idea is to change point of view with respect to
the usual finite dimensional setting as presented in
\cite{Bovier_denHollander_book,BEGK}  and to regard capacities and partition
functions as expectations of random variables under Gaussian measures, which are
well-defined in infinite dimension. This idea already appeared in
\cite{DiGesuLePeutrec_15} for the analysis of metastability of the Allen-Cahn
equation in space dimension $d=1$. In the present setting ($d=2$) the new point
of view is particularly powerful, since expectations under Gaussian measures can
be estimated using Wick calculus, and in particular the classical Nelson
argument \cite{Nelson73} which permits to bound expectations of exponentials of
Hermite polynomials. Another key argument is the observation from
\cite{berglund2007metastability} that the field $\phi$ can be decomposed into
its average and fluctuating part and that the (non-convex) potential $V$ is
convex in the transverse directions (see Lemma~\ref{lem:wo3}). An
additional key idea, following \cite[Section 3.2]{DiGesuLePeutrec_15}, is to
avoid using Hausdorff--Young inequalities in the discussion of Laplace
asymptotics (see \cite{BarretBovierMeleard,berglund2013sharp}) and rather use
Taylor expansions and global estimates, which lead to much better error
estimates, both in $\eps$ and in $N$. The rest of this paper is structured as
follows: in Sections~\ref{sec_results} and \ref{sec_outline} we  give the
precise assumptions, state our main theorem and give the necessary
background from potential theory. Lower bounds on the expected transition time
are proved in Section~\ref{sec_lb}, upper bounds are proved in
Section~\ref{sec_ub}. Some well-known facts about
Hermite polynomials and Wick powers are collected in Appendix~\ref{sec_Wick}.

\medskip

\noindent{\bf Acknowledgements:} The authors gratefully acknowledge the
hospitality and financial support they received during reciprocal visits at the
University of Warwick and the Universit\'e d'Orl\'eans.  GDG gratefully
acknowledges the support of the European Research Council under the European
Union's Seventh Framework Programme (FP/2007-2013) / ERC Grant Agreement number
614492. HW gratefully acknowledges support from the EPSRC through an EPSRC First
Grant and the Royal Society through a University Research Fellowship.
Finally, the authors would like to thank the two anonymous referees for
their careful reading of the first version of this manuscript, which led to
substantial improvements in the presentation.


\section{Results}
\label{sec_results}

Let $\T^2 = \R^2/(L\Z)^2$ denote the two-dimensional torus of size $L\times
L$. We are interested in the renormalised Allen--Cahn equations
\begin{equation}
 \partial_t \phi = \Delta \phi + \bigbrak{1+3 \eps C_N} \phi - \phi^3+
\sqrt{2\eps} \,\xi_N
\end{equation} 
for $\phi=\phi(t,x) : \R^+ \times \T^2 \to \R$, where $\xi_N$ approximates
space-time white noise on the scale $1/N$. In fact, we will consider spectral
Galerkin approximations of the above equation. Let $e_k(x) =
L^{-1}\e^{\icx(2\pi/L) k\cdot x}$ denote $L^2$-normalised Fourier basis
vectors of $L^2(\T^2)$, where $k\in\Z^2$ is the wave vector. Denote by $P_N$ the
projection on Fourier modes with wave number $k$ satisfying $\abs{k} =
\max\{\abs{k_1},\abs{k_2}\}\leqs N$, that is,
\begin{equation}
(P_N\phi)(x) = \sum_{k\in\Z^2 \colon \abs{k}\leqs N} 
 \langle \phi, e_k \rangle e_k(x)\;.
\end{equation} 
Then we consider the sequence of equations  
\begin{equation}
 \label{eq:Allen-Cahn} 
 \partial_t \phi = \Delta \phi + \bigbrak{1+3 \eps C_N} \phi - P_N \phi^3+
\sqrt{2\eps} \,\xi_N
\end{equation} 
where $\xi_N = P_N\xi$ is the spectral Galerkin approximation of space-time
white noise $\xi$.

We will assume periodic boundary conditions (b.c.), with a domain
size satisfying $0<L<2\pi$. This assumption guarantees that the identically zero
function plays the role of the transition state, which separates the basins
of attraction of the two stable solutions $\phi_\pm=\pm1$ of the deterministic
Allen-Cahn equation. 

Note that the deterministic system (take $\eps=0$ in~\eqref{eq:Allen-Cahn})
is a gradient system, with potential 
\begin{equation}
\label{eq:V_phi} 
 V_N[\phi] 
 = \frac12 \int_{\T^2} \bigbrak{\norm{\nabla \phi(x)}^2-\phi^2(x)} \6x 
 + \frac14
\int_{\T^2} \bigbrak{\phi^4(x)-6\eps C_N\phi^2(x) + 3\eps^2C_N^2} \6x\;.
\end{equation}
The measure $\e^{-V_N/\eps}$ is an invariant, reversible measure of the
stochastic system~\eqref{eq:Allen-Cahn}, and we will denote by $\cZ_N(\eps)$
its normalisation (the partition function of the system). 
The constant term $3\eps^2C_N^2$ in the second integral is of course irrelevant
for the dynamics, but it will simplify notations. This is related to the fact
that 
\begin{equation}
 \phi^4(x)-6\eps C_N\phi^2(x) + 3\eps^2C_N^2 
 = H_4(\phi(x),C_N)
\end{equation} 
is the so-called Wick renormalisation of $\phi^4(x)$ with respect to the
centered Gaussian measure  having covariance $\abs{P_N[-\Delta-1]^{-1}}$,
often denoted $\Wick{\phi^4(x)}$, where $H_4$ is the fourth Hermite
polynomial (see Section~\ref{sec_lb} and Appendix~\ref{sec_Wick}). The
renormalisation constant $C_N$ is given by 
\begin{equation}
\label{eq:cN} 
 C_N = \frac{1}{L^2} \Tr\bigpar{\abs{P_N[-\Delta-1]^{-1}}}
 := \frac{1}{L^2}\sum_{k\in\Z^2 \colon \abs{k}\leqs N}
\frac{1}{\abs{\lambda_k}}
\end{equation} 
where $\lambda_k = (2\pi/L)^2(k_1^2 +k_2^2)-1$.  Therefore, $C_N$ diverges
logarithmically as 
\begin{equation}
 \label{eq:cNlog}
 C_N \asymp \frac{2\pi}{L^2}\log(N)\;.
\end{equation} 
The choice of $C_N$ is somewhat arbitrary, as adding a constant independent of
$N$ to $C_N$ will also yield a well-defined limit equation as $N\to\infty$. See
Remark~\ref{rem:CN} below for the effect of such a shift on the results.  

\begin{figure}[tb]
\begin{center}
\begin{tikzpicture}
[>=stealth',x=3cm,y=3cm,main node/.style={draw,circle,fill=white,minimum
size=3pt,inner sep=0pt}]

\newcommand*{\del}{0.2}
\newcommand*{\h}{0.25}
\newcommand*{\hh}{0.15}
\pgfmathsetmacro{\myrho}{1 - \del}

\path[-,fill=green!70!blue!50] ({-\myrho},{\hh}) -- ({\myrho},{\hh})
-- ({\myrho},{-\hh}) -- ({-\myrho},{-\hh}) -- ({-\myrho},{\hh});

\path[-,fill=blue!50] ({-1-\del},{\h}) -- ({-1+\del},{\h}) -- ({-1+\del},{-\h})
 -- ({-1-\del},{-\h}) -- ({-1-\del},{\h});
 
 \path[-,fill=blue!50] ({1-\del},{\h}) -- ({1+\del},{\h}) -- ({1+\del},{-\h})
 -- ({1-\del},{-\h}) -- ({1-\del},{\h});

\draw[green!70!blue,thick] ({-\myrho},{\hh}) -- ({\myrho},{\hh}) --
({\myrho},{-\hh})
 -- ({-\myrho},{-\hh}) -- ({-\myrho},{\hh});

\draw[blue,thick] ({-1-\del},{\h}) -- ({-1+\del},{\h}) -- ({-1+\del},{-\h})
 -- ({-1-\del},{-\h}) -- ({-1-\del},{\h});

\draw[blue,thick] ({1-\del},{\h}) -- ({1+\del},{\h}) -- ({1+\del},{-\h})
 -- ({1-\del},{-\h}) -- ({1-\del},{\h});
 
\draw[->,thick] (-1.6,0) -- (1.6,0);
\draw[->,thick] (0,-0.6) -- (0,0.6);

\node[main node, semithick] at (-1,0) {}; 
\node[main node, semithick] at (1,0) {}; 

\node[] at (1.45,-0.1) {$\bar\phi$};
\node[] at (0.1,0.45) {$\phi_\perp$};
\node[] at (-1,-0.1) {$-L$};
\node[] at (1,-0.1) {$L$};

\node[blue] at (-1,{\h + 0.1}) {$A$};
\node[blue] at (1,{\h + 0.1}) {$B$};
\node[green!70!blue] at ({0.5*\myrho},{\hh + 0.1}) {$D$};

\draw[<->,blue,semithick] ({-1-\del},{-\h-0.15}) -- ({-1+\del},{-\h-0.15});
\node[blue] at (-1,{-\h - 0.23}) {$2\delta$};

\draw[<->,green!70!blue,semithick] (0,{-\hh-0.15}) -- ({1-\del},{-\hh-0.15});
\node[green!70!blue] at ({0.5*\myrho},{-\hh -0.23}) {$\rho$};

\end{tikzpicture}
\end{center}
\vspace{-3mm}
\caption[]{Geometry of the neighbourhoods $A$ and $B$ of the
deterministic stable solutions $\phi_\pm(x)$. The set $D$ will be needed
later in the proof, cf.~Section~\ref{ssec_capacity_lower}.}
\label{fig:lowerbound} 
\end{figure}
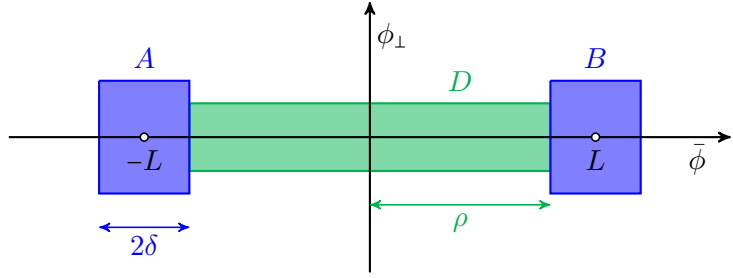

In the deterministic case $\eps=0$, the Allen--Cahn
equation~\eqref{eq:Allen-Cahn} has two stable stationary solutions given by
$\phi_-(x) = -1$  and $\phi_+(x) = 1$. We are interested in obtaining sharp
asymptotics on the expectation of the first-hitting time $\tau_B$ of a
neighbourhood $B$ of $\phi_+$, when starting near
$\phi_-$. The neighbourhood $B$ should have a minimal size. More precisely, 
we decompose any function $\phi:\T^2\to\R$ into its mean and oscillating part
by setting 
\begin{equation}
 \label{eq:phi_dec}
 \phi(x) = \bar\phi + \phi_\perp(x)
\end{equation} 
where $\bar\phi = L^{-2} \int_{\T^2} \phi(x)\6x$ and the
integral of $\phi_\perp$ over $\T^2$ is zero. Then we define two symmetric sets
$A$ and $B$ as follows (see~\figref{fig:lowerbound}).

\begin{definition}
\label{def:Dperp} 
Fix constants $\delta\in(0,1)$ and $s<0$, and let $D_\perp$ be a ball of
radius $r\sqrt{\eps\log(\eps^{-1})}$ in the Sobolev space $H^s(\T^d)$, 
where $r$ is a sufficiently large numerical constant. Then
\begin{align}
\nonumber
 A &= \bigsetsuch{\phi\in H^s(\T^d)}{\bar\phi \in [-1-\delta,-1+\delta], 
 \phi_\perp\in D_\perp}\;,\\
 B &= \bigsetsuch{\phi\in H^s(\T^d)}{\bar\phi \in [1-\delta,1+\delta], 
 \phi_\perp\in D_\perp}\;.
 \label{eq:AB}
\end{align}
\end{definition}

\begin{remark}
The definition of $D_\perp$ ensures that $A\cup B$ contains most of the mass
of the invariant probability measure $\cZ_N^{-1}\e^{-V_N/\eps}$ of the equation.
In fact, the result remains true if we replace $D_\perp$ by any sufficiently
regular set containing $D_\perp$, see Lemma~\ref{lem:cap_lb02}. 
\end{remark}

Our main result for periodic b.c.\ is as follows (recall from the
introduction that $\nu_k = (2\pi/L)^2(k_1^2 +k_2^2)+2 = \lambda_k+3$). 

\begin{theorem}
\label{thm:periodic_bc}
Assume $0< L < 2\pi$. There exists a sequence $\set{\mu_N}_{N\geqs1}$ of
probability measures concentrated on $\partial A$ such that in the case of
periodic b.c.,
\begin{align}
\nonumber
 \limsup_{N\to\infty} \bigexpecin{\mu_N}{\tau_B}
 &\leqs \frac{2 \pi}{|\lambda_0|} \sqrt{\prod_{k \in \Z^2} 
\frac{\abs{\lambda_k}}{\nu_k}
\,\exp\biggset{\frac{\nu_k - \lambda_k}{\abs{\lambda_k}}}}
\e^{[V(\phi_0)-V(\phi_-)]/\eps}
\bigbrak{1 + c_+\sqrt{\eps}\,}\;,\\
  \liminf_{N\to\infty} \bigexpecin{\mu_N}{\tau_B}
 &\geqs \frac{2 \pi}{|\lambda_0|} \sqrt{\prod_{k \in \Z^2} 
\frac{\abs{\lambda_k}}{\nu_k}
\,\exp\biggset{\frac{\nu_k - \lambda_k}{\abs{\lambda_k}}}}
\e^{[V(\phi_0)-V(\phi_-)]/\eps} 
\bigbrak{1 - c_-\eps}\;,
\label{eq:EK_periodic}
\end{align} 
where the constants $c_\pm$ are uniform in $\eps$.
\end{theorem}

Since  
$\nu_k = \lambda_k+3$ and $V(\phi_0)-V(\phi_-) = L^2/4$,
the leading term in~\eqref{eq:EK_periodic} can also be written as 
\begin{equation}
 2\pi
 \biggl( \frac{\e^{3/\abs{\lambda_0}}}
 {\abs{\lambda_0}(\lambda_0+3)}
 \prod_{k\neq 0} 
 \frac{\e^{3/\lambda_k}}{1+3/\lambda_k} \biggr)^{1/2}
\e^{L^2/4\eps}\;.
\end{equation}
The infinite product indeed converges, since 
\begin{equation}
 \log \biggl( \frac{\e^x}{1+x} \biggr) = x - \log(1+x) \leqs \frac12 x^2\;,
\end{equation} 
and the sum over $\Z^2$ of $\lambda_k^{-2}$ converges, unlike the sum of
$\lambda_k^{-1}$ that would arise without the regularising term
$\e^{3/\lambda_k}$. On a more abstract level, as already mentioned in the
introduction, this is due to the fact that we have replaced the usual
Fredholm determinant $\det(\mathrm{Id}+T)$ (with
$T=3\abs{(-\Delta-1)^{-1}}$) by the Fredholm--Carleman determinant
\begin{equation}
 \label{eq:det2}
 {\det}_2(\mathrm{Id}+T) = \det(\mathrm{Id}+T)\e^{-\Tr T}\;,
\end{equation} 
which is defined for every Hilbert--Schmidt perturbation $T$ of the identity, 
without the requirement of $T$ to be
trace-class~\cite[Chapter~5]{Simon_Trace_ideals}.

\begin{remark}
\label{rem:Neumann}
An analogous result holds for zero-flux Neumann b.c.\ $\partial_{x_i}
\phi(t,x) = 0$ whenever $x_i\in\set{0,L}$ for $i=1$ or $2$, provided
$0<L<\pi$. The only difference is that one has to replace $2\pi/L$ by
$\pi/L$ in
the definition of the $\lambda_k$, and that the sums are over $k\in\N_0^2$
instead of $k\in\Z^2$. 
\end{remark}

\begin{remark}
\label{rem:CN}
The definition~\eqref{eq:cN} of the renormalisation constant $C_N$ is not
unique, and one would still obtain a well-defined limit
for~\eqref{eq:Allen-Cahn} if $C_N$ were replaced by $C_N + \theta/L^2$ for some
constant $\theta\in\R$ (or even by $C_N + \theta_N/L^2$, where $\theta_N$
converges to a limit $\theta$ as $N\to\infty$). One easily checks that the
effect of such a shift in the renormalisation constant is to multiply the
expected transition time by a factor $\e^{3\theta/2\abs{\lambda_0}}$. 
\end{remark}

\begin{remark}
\label{rem:initial_cond}
The measures $\mu_N$ appearing in Theorem~\ref{thm:periodic_bc} are the
normalised equilibrium measures $\mu_{A,B}$ of the finite dimensional
dynamics of $\phi_N$. This measure appears in the result,
because the method of proof is based on an asymptotic analysis as $\eps \to 0$
of the exact formula \eqref{eq:expec_tau} for $\expecin{\mu_N}{\tau_B}$ which is
available with this initial distribution. One would expect that the precise
starting distribution does not affect the final result and that equation
\eqref{eq:EK_periodic} remains true with $\mu_N$ replaced by any distribution
which is supported on a sufficiently small neighbourhood of $\phi_-$, in
particular a deterministic initial condition at $\phi_-$. For the
one-dimensional SPDE such a result was indeed obtained in
\cite{berglund2013sharp}, by post-processing a bound involving the equilibrium
measure, building on a coupling technique developed in
\cite{Martinelli_Olivieri_Scoppola_89}. It would be interesting to see if a
similar post-processing is also possible in the context of our renormalised 
SPDE.
\end{remark}

\begin{remark}
\label{rem:limit}
In a similar spirit, it might seem more natural to state the bound
\eqref{eq:EK_periodic} in the limit $N \to \infty$, i.e. for solutions of the
full renormalized SPDE rather than stating a uniform-in-$N$ bound over the
Galerkin approximations. In the one-dimensional case, such a result was
obtained in~\cite[Proposition~3.4]{berglund2013sharp}, using various a priori
estimates available for $d=1$.
Using the technique from~\cite{daPratoDebussche} it is
straightforward to show that the approximate solutions $\phi_N$ (say all
starting in $\phi_-$) and the corresponding stopping times $\tau_N$ converge in
probability as $N \to \infty$. One can then pass  to the limit in the
expectations $\expecin{\phi_-}{\tau_N}$ (for $N \to \infty$ keeping  $\eps$
fixed) provided the random variables $\tau_N$ are uniformly integrable, which is
implied, for example, by a bound of the type $\sup_N \expecin{\phi_-}{\tau_N^p}<
\infty$ for some $p>1$. We expect that a much stronger bound of the type
$\sup_{N} \expecin{\phi_-}{\exp( \lambda \tau_N )} < \infty$ for $\lambda$ small
enough can be obtained, using the results derived in~\cite{Tsatsoulis-Weber}.
\end{remark}

%

\begin{remark}
\label{rem:epsilon}
The error term in $\sqrt{\eps}$ in the upper bound for the expected transition
time is due to our using less sharp approximations in the Laplace asymptotics.
It is in principle possible, as was done in the one-dimensional case
in~\cite{DiGesuLePeutrec_15}, to derive further terms in the asymptotic
expansion. In particular it is expected that the leading error term has
order~$\eps$.
\end{remark}


\section{Some potential theory}
\label{sec_outline}

When considering the spectral Galerkin approximations~\eqref{eq:Allen-Cahn},
it will sometimes be convenient to use Fourier variables $z_k$ defined by
\begin{equation}
 \label{eq:Fourier} 
 \phi_N(t,x) = \sum_{k\in\Z^2 \colon \abs{k}\leqs N} z_k(t) e_k(x)\;.
\end{equation}
In order to ensure that $\phi_N$ is real-valued, the
coefficients $z_k$ are chosen to take values in
\begin{align}
 \{ (z_k) \in \C^{(-N, \ldots, N)^2} \colon z_{-k} = \overline{z_{-k}}  \quad
\text{for all } k \} 
\end{align}
which we identify with $\R^{(2N+1)^2}$ throughout. In particular, we will always
interpret gradients
 and integration with respect to Lebesgue measure $\6 z$ in terms of this
identification.
The spectral Galerkin approximation with cut-off $N$ is equivalent to the
system of It\^o SDEs 
\begin{equation}
\label{eq:Ito} 
\6z(t) = -\nabla V_N(z(t))\6t + \sqrt{2\eps} \6W_t\;,
\end{equation} 
where the potential, obtained by evaluating~\eqref{eq:V_phi} in $\phi_N$, is
given by 
\begin{equation}
 V_N(z) = \frac12 \sum_{\abs{k}\leqs N} 
 \lambda_k \abs{z_k}^2 + \frac{1}{4}
\Big( \frac{1}{L^2} \sum_{\substack{k_1 + k_2 +k_3 +k_4=0 \\  \abs{k_i} \leqs
N}} 
 z_{k_1} z_{k_2} z_{k_3} z_{k_4} - 6 \eps C_N \sum_{\abs{k}\leqs N}|z_k|^2
 + 3L^2 \eps^2 C_N^2 \Big)\;.
\label{eq:V_N} 
\end{equation} 
Arguments based on potential theory (see~\cite{BEGK} and~\cite[Corollary
7.30]{Bovier_denHollander_book}) show that for any finite $N$ one has the
relation 
\begin{equation}
 \label{eq:expec_tau} 
 \bigexpecin{\mu_{A,B}}{\tau_B} 
 = \frac{1}{\capacity_A(B)} \int_{B^c} h_{A,B}(z) \e^{-V_N(z)/\eps}\6z\;. 
\end{equation} 
Here $h_{A,B}(z)$ is the \emph{committor function} (or \emph{equilibrium
potential}) 
\begin{equation}
 \label{eq:committor} 
 h_{A,B}(z) = \bigprobin{z}{\tau_A < \tau_B}\;,
\end{equation} 
where $\tau_A$ denotes the first-hitting time of a set $A\subset\R^{(2N+1)^2}$.
The term $\capacity_A(B)$ is the so-called \emph{capacity}, which admits several
equivalent expressions:
\begin{align}
\label{eq:cap1} 
\capacity_A(B) 
&= \eps \int_{(A\cup B)^c} \norm{\nabla h_{A,B}(z)}^2 \e^{-V_N(z)/\eps}\6z \\
\label{eq:cap2} 
&= \eps \inf_{h\in\cH_{A,B}}\int_{(A\cup B)^c} \norm{\nabla h}^2
\e^{-V_N(z)/\eps}\6z \\
\label{eq:cap3} 
&= \int_{\partial A} \e^{-V_N(z)/\eps} \rho_{A,B}(\6z)\;,
\end{align}
where $\cH_{A,B}$ is the set of functions $h\in H^1$ such that $h=1$ in $A$ and
$h=0$ in $B$, and $\rho_{A,B}(\6z)$ is a measure concentrated on $\partial A$,
called the \emph{equilibrium measure}. The density of this measure (with
respect to the surface measure on $\partial A$) is equal to the exterior normal
derivative of the committor function $h_{A,B}$. Finally, $\mu_{A,B}$ is the
probability measure on $\partial A$ obtained by normalising
$\rho_{A,B}\e^{-V/\eps}$:
\begin{equation}
\label{eq:nuAB} 
\mu_{A,B}(\6z) = \frac{1}{\capacity_A(B)} \e^{-V_N(z)/\eps} \rho_{A,B}(\6z)\;.
\end{equation} 

The following symmetry argument allows us to link the expected transition time
to the partition function of the system. 

\begin{lemma}
If $A$ and $B$ are symmetric with respect to the plane $z_0=0$ then 
\begin{equation}
 \label{eq:hAB_sym} 
 \int_{B^c} h_{A,B}(z) \e^{-V_N(z)/\eps} \6z 
= \frac12 \int_{\R^{(2N+1)^2}}  \e^{-V_N(z)/\eps} \6z
=: \frac12 \cZ_N(\eps)\;.
\end{equation} 
\end{lemma}
\begin{proof}
Consider the reflection $S$ given by 
\[
 S(z_0,z_\perp) = (-z_0,z_\perp)\;.
\]
The potential $V_N$ satisfies the symmetry 
\[
 V_N(Sz) = V_N(z)
\]
which implies 
\begin{equation}
 \label{eq:VN_sym} 
 \int_{\set{z_0<0}} \e^{-V_N(z)/\eps} \6z =
\int_{\set{z_0>0}} \e^{-V_N(z)/\eps} \6z = 
\frac12 \cZ_N(\eps)\;. 
\end{equation} 
Assuming we choose $A$ and $B$ such that $B=SA$, the committor satisfies 
\[
 h_{A,B}(z) = h_{B,A}(Sz)\;.
\]
In addition, we always have 
\[
 h_{A,B}(z) = 1 - h_{B,A}(z)\;.
\]
Now observe that we have 
\begin{align*}
\int_{\R^{(2N+1)^2}} &h_{A,B}(z) \e^{-V_N(z)/\eps} \6z \\
&= \int_{\set{z_0<0}} h_{A,B}(z) \e^{-V_N(z)/\eps} \6z 
+ \int_{\set{z_0>0}} (1-h_{B,A}(z)) \e^{-V_N(z)/\eps} \6z \\
&= \int_{\set{z_0<0}} h_{A,B}(z) \e^{-V_N(z)/\eps} \6z 
+ \int_{\set{z_0>0}} \e^{-V_N(z)/\eps} \6z 
- \int_{\set{z_0<0}} h_{A,B}(z) \e^{-V_N(z)/\eps} \6z \\
&= \int_{\set{z_0>0}} \e^{-V_N(z)/\eps} \6z\;,
\end{align*}
and the conclusion follows from~\eqref{eq:VN_sym}. 
\end{proof}

As a consequence of~\eqref{eq:hAB_sym}, we can rewrite~\eqref{eq:expec_tau} in
the form
\begin{equation}
 \label{eq:expec_tau1}
 \bigexpecin{\mu_{A,B}}{\tau_B} 
 = \frac{1}{2\capacity_A(B)} \cZ_N(\eps)\;. 
\end{equation} 
Note that this relation can also be written as 
\begin{equation}
 \label{eq:expec_tau2}
 \frac{1}{\bigexpecin{\mu_{A,B}}{\tau_B}}
 = 2\bigexpecin{\pi_N(\eps)}{\eps \norm{\nabla h_{A,B}}^2}\;,
\end{equation} 
where $\pi_N(\eps)$ is the probability measure on $\R^{(2N+1)^2}$ with density
$\cZ_N(\eps)^{-1}\e^{-V_N(z)/\eps} \6z$. 

The result will follow if we are able to prove the estimate
\begin{equation}
 \label{eq:cap_estimate} 
  \capacity_A(B) = \sqrt{\frac{\abs{\lambda_0}\eps}{2\pi}}
 \prod_{0<\abs k\leqs N} \sqrt{\frac{2\pi\eps}{\lambda_k}}
 \,\bigbrak{1 + r(\eps)}
\end{equation} 
on the capacity with $-c_-\sqrt{\eps} \leqs r(\eps) \leqs c_+\eps$, as well as
an estimate on the partition function of the form 
\begin{equation}
 \label{eq:ZN_estimate}
 \frac12\cZ_N(\eps) 
  = \prod_{\abs k\leqs N} \sqrt{\frac{2\pi\eps}{\lambda_k+3}}
 \e^{L^2/4\eps} \e^{3 L^2 C_N/2} 
 \,\bigbrak{1 - r(\eps)}\;.
\end{equation} 
Here the exponent $L^2/4$ is the value of the original
potential~\eqref{eq:potential} at the stable stationary solution $\phi_-$,
while the term $3 L^2 C_N/2$ is due to the renormalisation, which makes the
potential well deeper and deeper as $N\to\infty$. Note that owing to the
expression~\eqref{eq:cN} of $C_N$, this is exactly the extra term transforming
the divergent Fredholm determinant into a converging Carleman--Fredholm
determinant.


\section{Lower bound on the expected transition time}
\label{sec_lb}

We will start by deriving a lower bound on the expected transition time,
i.e., we will prove the upper bound on the capacity \eqref{eq:cap_estimate}
(in Section~\ref{ssec_capacity_upper}) and the lower bound on the partition
function \eqref{eq:ZN_estimate}  (in Section~\ref{ssec_ZN_lower}). These bounds
are  somewhat simpler to obtain than the matching lower and upper bound in
Section~\ref{sec_ub}.

For the upper bound on $\capacity_A(B)$ we choose a specific test-function $h_+$
(defined in \eqref{eq:cap_ub01}) and plug it into the variational
characterisation \eqref{eq:cap3} of the capacity. The specific form of $h_+$ is
similar to the one used in \cite{BEGK}. It is given by (an approximation to) the
explicit minimiser of the one-dimensional version of \eqref{eq:cap3} in the
$z_0$ direction and chosen to be constant in all transversal directions. Then an
explicit calculation with this test function leads to an upper bound on the
capacity which consists of the desired pre-factor multiplied by the Gaussian
expectation of the exponential of  a negative fourth Wick power (see
equation~\eqref{eq:cap_ub04} below). Here the  renormalisation introduces a
difficulty, because the Wick power is not bounded from below uniformly in $N$.
In order to analyse this expression, we invoke techniques from constructive
Quantum Field Theory, more precisely, the Nelson argument  (see
\cite{Nelson73}). This is  implemented in two steps in
Proposition~\ref{prop:nelson1} and Proposition~\ref{prop:nelson2}
leading to the desired bound.

In order to derive the lower bound on the $\cZ_N(\eps)$ we use the symmetry of
the system to reduce to the integration over a half space. We then recentre the
field around the minimiser $1$ and perform the corresponding change of
coordinates in the expression for $V_N$. At this point we use the binomial
identity for Wick powers (see \eqref{A3-1}) as well as the transformation rules
for changing the \lq\lq mass\rq\rq\ in a Wick power (see \eqref{A5-1}),
resulting again in a Gaussian expectation of the exponential of a quartic Wick
polynomial. This time the Gaussian covariance changes with respect to
Section~\ref{ssec_capacity_upper} and the Wick polynomial is defined with
respect to this modified measure and also has a non-trivial cubic term. The
\lq\lq renormalisation\rq\rq\ factor  $\e^{3 L^2 C_N/2} $ in
\eqref{eq:ZN_estimate} appears at this point, because the change of Gaussian
measure is given by a quadratic Wick power which produces an extra term in the
Gaussian normalisation constant. Finally, the Gaussian expectation is treated
with a relatively simple argument using Jensen's inequality in
Proposition~\ref{prop:ZN_lb}.
%



\subsection{Upper bound on the capacity}
\label{ssec_capacity_upper}

One can obtain an upper bound on the capacity by inserting any test function in
the right-hand side of~\eqref{eq:cap2}. Let $\delta>0$ be a small constant and
define 
\begin{equation}
 \label{eq:cap_ub01} 
 h_+(z) = 
 \begin{cases}
 1 & \text{if $z_0 \leqs -\delta$\;,} \\
 \dfrac{\displaystyle \int_{z_0}^\delta
\e^{-\abs{\lambda_0}t^2/2\eps}\6t}{
\displaystyle \int_{-\delta}^\delta
\e^{-\abs{\lambda_0}t^2/2\eps}\6t}
 & \text{if $-\delta < z_0 < \delta$\;,} \\
 0 & \text{if $z_0 \geqs \delta$\;.}
 \end{cases}
\end{equation} 
Although $\abs{\lambda_0}=1$, we will keep $\lambda_0$ in the notation as it
allows to keep track of its influence on the result. Observe that 
\begin{equation}
\label{eq:cap_ub02}
\norm{\nabla h_+(z)}^2 = 
\begin{cases}
 \dfrac{\e^{-\abs{\lambda_0}z_0^2/\eps}}{ \displaystyle
 \biggpar{\int_{-\delta}^\delta 
 \e^{-\abs{\lambda_0}t^2/2\eps}\6t}^2} 
 & \text{if $-\delta < z_0 < \delta$\;,} \\
 0 & \text{otherwise\;.}
\end{cases}
\end{equation} 
Note that 
\begin{equation}
 \label{eq:cap_ub03} 
 \biggpar{\int_{-\delta}^\delta \e^{-\abs{\lambda_0}t^2/2\eps}\6t}^2 =
\frac{2\pi\eps}{\abs{\lambda_0}}
\,\bigbrak{1+\Order{\e^{-\delta^2/2\eps}}}\;,
\end{equation} 
where here and throughout the paper, the notation $f(\eps) =
\Order{g(\eps)}$ indicates that there exists $\eps_0>0$ such that $|f(\eps)|$ is
bounded by a constant times $|g(\eps)|$ uniformly in $\eps\in(0,\eps_0)$. Note
that here the parameter $\delta$ is small but fixed.
Inserting~\eqref{eq:cap_ub03} in~\eqref{eq:cap2} we get 
\begin{align}
\nonumber
 \capacity_A(B) 
 &\leqs \frac{\abs{\lambda_0}}{2\pi} 
 \int_{\R^{(2N+1)^2}} \e^{-[V_N(z)+\abs{\lambda_0}z_0^2]/\eps} \6z
 \,\bigbrak{1+\Order{\e^{-\delta^2/2\eps}}} \\
 &= \frac{\abs{\lambda_0}}{2\pi} \eps^{\frac12(2N+1)^2}
 \int_{\R^{(2N+1)^2}} \e^{-[V_N(\sqrt{\eps}y)/\eps+\abs{\lambda_0}y_0^2]} \6y
 \,\bigbrak{1+\Order{\e^{-\delta^2/2\eps}}}\;.
 \label{eq:cap_ub04} 
\end{align} 
Using the scaling $z=\sqrt{\eps}y$ in~\eqref{eq:V_N}, which amounts to
working with the scaled field
\begin{equation}
\label{eq:phi_N_hat} 
\phi_N=\sqrt{\eps}\hat\phi_N\;,
\end{equation} 
shows that the exponent can be written in the form  
\begin{equation}
 \label{eq:cap_ub05}
 \frac{1}{\eps} \bigbrak{V_N(\sqrt{\eps}y) + \eps\abs{\lambda_0}y_0^2}
 = g_N(y) + \eps w_N(y)\;,
\end{equation} 
where
\begin{align}
\nonumber
 g_N(y) &= \frac12 \abs{\lambda_0} y_0^2 
 + \frac12 \sum_{0<\abs{k}\leqs N}\lambda_k \abs{y_k}^2\;, \\
 w_N(y) &=  \frac14 \int_{\T^2} \bigl(\hat\phi_N^4(x) - 6 C_N \hat\phi_N^2(x) 
 + 3 C_N^2\bigr) \6x \;.
 \label{eq:cap_ub06}
\end{align} 
The quadratic form $g_N$ allows us to define a Gaussian probability measure 
$\6\gamma(y) = \cN^{-1}\e^{-g(y)}\6y$, with normalisation 
\begin{equation}
 \label{eq:cap_ub07}
 \cN =  \sqrt{\frac{2\pi}{\abs{\lambda_0}}} \prod_{k\neq0, \abs{k}\leqs N}
\sqrt{\frac{2\pi}{\lambda_k}}\;.
\end{equation} 
We can thus rewrite the upper bound~\eqref{eq:cap_ub04} in the form
\begin{equation}
\label{eq:cap_ub08}
\capacity_A(B) \leqs 
\sqrt{\frac{\abs{\lambda_0}\eps}{2\pi}} \prod_{k\neq0, \abs{k}\leqs N}
\sqrt{\frac{2\pi\eps}{\lambda_k}} 
\bigexpecin{\gamma}{\e^{-\eps w_N}}\,\bigbrak{1+\Order{\e^{-\delta^2/2\eps}}}\;.
\end{equation} 
The term $\expecin{\gamma}{\e^{-\eps w_N}}$ can be estimated using the Gaussian
calculus developed in Appendix~\ref{sec_Wick}. Indeed, the law of the field
$\hat\phi_N$ under $\gamma$ is exactly as described 
there. Furthermore, $C_N = \E^{\gamma} [\hat\phi_N(x)^2]$ for each $x \in \T^2$
so that the term $w_N(y)$, defined in \eqref{eq:cap_ub06}, can be rewritten as
the integral over the fourth Wick power of $\hat{\phi}$ with respect to this
Gaussian measure, that is, 

\begin{equation}
 \label{eq:cap_ub09}
  w_N  = \frac{1}{4} \int_{\T^2} H_4(\hat\phi_N,C_N) \6 x ,
\end{equation}
(where $H_4(X,C) = X^4 - 6 CX^2 +3C^2$, see \eqref{e:def-Hermite1} for the definition of 
the Hermite polynomial $H_n$).
In particular $w_N$ has zero mean under the Gaussian measure $\6\gamma$ and
according to \eqref{A13} all its stochastic moments are uniformly bounded in
$N$.

We now derive a uniform-in-$N$ bound on $\bigexpecin{\gamma}{\e^{- w_N}}$
following a classical argument 
due to Nelson (see e.g. \cite[Sec 8.6]{Glimm_Jaffe_81} or \cite[Sec.
4]{da2007wick}).

\begin{prop}
\label{prop:nelson1} 
There exists a constant $K$, independent of $N$, such that 
\begin{equation}
 \label{eq:nelson_prop_1}
 \bigexpecin{\gamma}{\e^{-w_N}} 
 \leqs K\;.
\end{equation}
\end{prop}
\begin{proof}
First note that the definition~\eqref{eq:A2} of Hermite polynomials implies
for any $M\in\N$
\[
H_4(\hat\phi_M,C_M) = (\hat\phi_M^2(x) - 3C_M)^2 - 6 C_M^2\;,  
\]
so that 
\[
 w_M \geqs -\frac32L^2 C_M^2 =: -D_M\;. 
\]
Since
$\expecin{\gamma}{\e^{-w_N}\indexfct{w_N \geqs 0}} \leqs
\probin{\gamma}{w_N \geqs 0}\leqs 1$, it is sufficient to bound 
\begin{align*}
\bigexpecin{\gamma}{\e^{-w_N}\indexfct{w_N < 0}} 
&= 1 + \int_0^\infty \e^t \bigprobin{\gamma}{-w_N > t}\6t \\
&\leqs \e + \int_1^\infty \e^t \bigprobin{\gamma}{-w_N > t}\6t \;.
\end{align*}
If $t \geqs D_N$, then $ \bigprobin{\gamma}{-w_N > t} =0$, otherwise
we have for any $M$
\begin{align*}
\bigprobin{\gamma}{-w_N > t} 
& \leqs \bigprobin{\gamma}{w_M - w_N > t - D_M} \\
& \leqs \bigprobin{\gamma}{\abs{w_M - w_N}^{p(t)} > \abs{t -
D_M}^{p(t)}} \;, 
\end{align*}
for any choice of $p(t)\in2\N$. We apply this inequality for $M=M(t)$
satisfying 
\begin{equation}\label{conD}
 t - D_{M(t)} \geqs 1\;,
\end{equation}
which implies that $M(t)<N$.

Then we get by Markov's inequality and  Lemma~\ref{le:Nelson} combined with
\eqref{A5} (cf.~\eqref{A14})
\begin{align*}
\bigprobin{\gamma}{-w_N > t} 
&\leqs \bigexpecin{\gamma}{\abs{w_{M(t)} -w_N}^{p(t)}} \\
&\leqs \const (p(t)-1)^{p(t)} \bigexpecin{\gamma}{\abs{w_{M(t)} -
w_N}^2}^{p(t)/2} \\
&\leqs \const \frac{(p(t)-1)^{p(t)}}{M(t)^{(1-\eta) p(t)/2}} \;
\end{align*}
for any $\eta>0$.
The condition \eqref{conD} on $M(t)$ imposes that $\log M(t)$ grows at most as
$t^{1/2}$. 
Choosing for instance $p(t) \sim t^\beta$ for $\beta>\frac12$, since 
\[
 \log \bigl( \e^t \bigprobin{\gamma}{-w_N > t} \bigr)
 \leqs t + \beta t^\beta \log t - c(1-\eta)t^{\beta+1/2}
\]
for a constant $c>0$ depending only on $L$, 
we obtain a convergent integral.
\end{proof}

This a priori estimate can now quite easily be turned into a sharper estimate.
Indeed, we have the following improvement. 

\begin{prop}
\label{prop:nelson2}
We have 
\begin{equation}
 \label{eq:nelson_prop_2}
 \bigexpecin{\gamma}{\e^{-\eps w_N}} 
 = 1 + \Order{\eps}\;,
\end{equation}
where the remainder is bounded uniformly in $N$. 
\end{prop}
\begin{proof}
Introduce the sets 
\[
 \Omega_+ = \bigsetsuch{\hat\phi_N}{w_N > 0}
\]
and $\Omega_-=\Omega_+^c$. Then we have 
\[
 \bigexpecin{\gamma}{\e^{-\eps w_N}\indicator{\Omega_+}}
 = \fP^{\gamma}(\Omega_+) 
 + \bigexpecin{\gamma}{(\e^{-\eps w_N}-1)\indicator{\Omega_+}}\;.
\]
Note that on $\Omega_+$, we have $-\eps \abs{w_N} \leqs \e^{-\eps
w_N}-1 \leqs 0$, so that 
\[
 \fP^{\gamma}(\Omega_+) - \eps
\bigexpecin{\gamma}{\abs{w_N}\indicator{\Omega_+}}
 \leqs \bigexpecin{\gamma}{\e^{-\eps w_N}\indicator{\Omega_+}}
 \leqs \fP^{\gamma}(\Omega_+)\;.
\]
Since $U_{4,N} $ has finite variance bounded uniformly in $N$, we know by
Cauchy--Schwarz
that $\bigexpecin{\gamma}{\abs{U_{4,N}}}$ is bounded uniformly in $N$.
Similarly, we have 
\[
 \bigexpecin{\gamma}{\e^{-\eps w_N}\indicator{\Omega_-}}
 = \fP^{\gamma}(\Omega_-) 
 + \bigexpecin{\gamma}{(\e^{-\eps w_N}-1)\indicator{\Omega_-}}\;.
\]
This time, we use that on $\Omega_-$, one has 
$0 \leqs \e^{-\eps w_N}-1 \leqs \eps \abs{w_N} \e^{-\eps
w_N}$. Thus by Cauchy--Schwarz, 
\begin{align*}
0 \leqs \bigexpecin{\gamma}{(\e^{-\eps w_N}-1)\indicator{\Omega_-}}
&\leqs 
\eps \bigexpecin{\gamma}{\abs{w_N} \e^{-\eps
w_N}\indicator{\Omega_-}} \\
&\leqs \eps \Bigpar{\bigexpecin{\gamma}{\e^{-2\eps
w_N}\indicator{\Omega_-}} \bigexpecin{\gamma}{
\abs{w_N}^2\indicator{\Omega_-}}}^{1/2}
\end{align*}
The term $\bigexpecin{\gamma}{\abs{w_N}^2}$ is bounded uniformly in $N$
as before, while the term $\bigexpecin{\gamma}{\e^{-2\eps
w_N}}$ is bounded uniformly in $N$ for $\eps\leqs2$ by
Proposition~\ref{prop:nelson1}. Summing the two estimates, we get the result.  
\end{proof}

Substituting this estimate in~\eqref{eq:cap_ub08}, we immediately get the
following upper bound on the capacity. 

\begin{cor}
\label{cor:cap_ub} 
There exists a constant $c_+$, uniform in $\eps$ and $N$, such that the capacity
satisfies the upper bound 
\begin{equation}
\label{eq:cap_ub20}
\capacity_A(B) \leqs 
\sqrt{\frac{\abs{\lambda_0}\eps}{2\pi}} \prod_{k\neq0, \abs{k}\leqs N}
\sqrt{\frac{2\pi\eps}{\lambda_k}} 
\,\bigbrak{1+c_+\eps}\;.
\end{equation}
\end{cor}


\subsection{Lower bound on the partition function}
\label{ssec_ZN_lower}

By symmetry, cf.~\eqref{eq:VN_sym}, the partition function can be computed
using the relation 
\begin{equation}
 \label{eq:ZN_lb01}
 \frac12 \cZ_N(\eps) = \int_{\Omega'_+} \e^{-V_N(z)/\eps}\6z\;, 
 \qquad
 \Omega'_+ = \set{z_0 > 0}\;.
\end{equation} 
A lower bound on $\cZ_N(\eps)$ can be obtained quite directly from Jensen's
inequality. It will be convenient to shift coordinates to the positive
stable stationary solution of the deterministic equation (without the
normalisation). That is, we set 
\begin{equation}
 \label{eq:ZN_lb02}
 \phi_N(x) = 1 + \sqrt{\eps} \hat\phi_{N,+}(x)\;,
\end{equation} 
with the Fourier decomposition 
\begin{equation}
 \label{eq:ZN_lb03}
 \hat\phi_{N,+}(x) = \sum_{\abs{k}\leqs N} y_k e_k(x)\;.
\end{equation} 
Let $\hat\Omega'_+ = \set{y_0 > -1/\sqrt{\eps}\,}$ denote the image of
$\Omega'_+$ under this transformation.
Substituting in~\eqref{eq:V_phi} and using the relation  \eqref{A3-1} yields the
following expression
for the potential:
\begin{align}
\nonumber
 V_N^+(y) :={}& \frac{1}{\eps} V_N[1 + \sqrt{\eps} \hat\phi_{N,+}(x)] \\
\nonumber
=& -\frac{L^2}{4\eps} + \frac12 \int_{\T^2} \Bigl(\norm{\nabla
\hat\phi_{N,+}(x)}^2 -
\hat\phi_{N,+}^2(x)  + 3 H_2(\hat\phi_{N,+}, C_N)\Bigr) \6 x \\
\label{eq:ZN_lb04}
&+ \frac{1}{4} \int_{\T^2} \Bigl(4  \sqrt{\eps} H_3(\hat\phi_{N,+}(x), C_N) +
\eps H_4 (\hat\phi_{N,+}(x), C_N)\Bigr) \6 x.
\end{align} 
Now the relevant Gaussian measure $\gamma_+$ is defined by the quadratic form 
\begin{align}
\notag
 g_{N,+}(y) &=   \frac12 \int_{\T^2} \Bigl( \norm{\nabla
\hat\phi_{N,+}(x)}^2 -
\hat\phi_{N,+}^2(x) 
+ 3 \hat\phi_{N,+}^2(x) \Bigr) \6 x\\
 \label{eq:ZN_lb05}
 &= \frac{1}{2} \sum_{0<\abs{k}\leqs N}(\lambda_k+3) \abs{y_k}^2\;.
\end{align} 
Observe that a term $-\frac32 C_N L^2$ appears owing to the Hermite
polynomial $ 3 H_2(\hat\phi_{N,+}, C_N)$. It is 
precisely this term which is ultimately responsible for the renormalisation of
the pre-factor. 
To bound expectations of the terms appearing in the last line of
\eqref{eq:ZN_lb04} it is convenient to rewrite 
them as Wick powers with respect to the Gaussian measure  defined by $g_{N,+}$.
The associated renormalisation constant is 
\begin{equation}
 \label{eq:ZN_lb06}
 C_{N,+} = \frac{1}{L^2} \sum_{0<\abs{k}\leqs N} \frac{1}{\lambda_k+3}\;.
\end{equation}

Observe in particular that 
\begin{equation}
 \label{eq:ZN_lb07}
 C_N - C_{N,+} = \frac{1}{L^2} \sum_{0<\abs{k}\leqs N}
\frac{3}{\abs{\lambda_k}(\lambda_k+3)}
\end{equation} 
is bounded uniformly in $N$. Using the relation \eqref{A5-1} that allows to
transform Hermite polynomials with respect to 
different constants we get
\begin{align}
\notag
  \sqrt{\eps} H_3(\hat\phi_{N,+}, C_N) 
  ={}&   \sqrt{\eps} H_3(\hat\phi_{N,+},
C_{N,+}) -
3\sqrt{\eps}(C_N-C_{N,+}) \hat\phi_{N,+} \\
  \notag
   \frac{\eps}{4} H_4 (\hat\phi_{N,+}, C_N) 
   ={}& \frac{ \eps}{4} H_4
(\hat\phi_{N,+},
C_{N,+}) - \frac{3}{2}\eps  (C_N-C_{N,+}) H_2 (\hat\phi_{N,+}, C_N) \\
   \label{e:ZN_LB1}
   &{}+ \frac{3}{4} \eps  (C_N-C_{N,+})^2\;.
\end{align}
Now we define the random variables 
\begin{equation}
 \label{eq:ZN_lb08}
 U_{n,N}^+ = \int_{\T^2} \Wick{\hat\phi_{N,+}^n(x)} \6x = \int_{\T^2}
H_n(\hat\phi_{N,+}(x), C_{N,+}) \6 x
\end{equation} 
which have zero mean under $\gamma_+$ as well as a variance bounded uniformly in
$N$.
Substituting~\eqref{e:ZN_LB1} in~\eqref{eq:ZN_lb04}, we get 
\begin{equation}
 \label{eq:ZN_lb09}
 V_N^+(y) = q + g_{N,+}(y) + w_{N,+}(y)\;,
\end{equation} 
where 
\begin{align}
q &= - \frac{L^2}{4 \eps} - \frac{3}{2}L^2C_N+  \frac34 L^2\eps(C_N - C_{N,+})^2
\nonumber\\
 w_{N,+}(y)
 &= \sqrt{\eps}\, U_{3,N}^+   + \frac14 \eps U_{4,N}^+ - 3 (C_N- C_{N,+} )
\big(\frac{\eps}{2} U_{2,N}^+ +\sqrt{\eps} U_{1,N}^+ \big) .
 \label{eq:ZN_lb10}
\end{align} 
It follows by a similar argument as in the previous section that 
\begin{equation}
 \label{eq:ZN_lb11}
 \frac12 \cZ_N(\eps) = \prod_{\abs{k}\leqs N}
\sqrt{\frac{2\pi\eps}{\lambda_k+3}} \e^{-q}
\bigexpecin{\gamma_+}{\e^{-w_{N,+}}\indicator{\hat\Omega'_+}}\;.
\end{equation} 

\begin{prop}
\label{prop:ZN_lb}
There exists a constant $c_-$, independent of $N$ and $\eps$, such that 
\begin{equation}
 \label{eq:ZN_lb12}
 \bigexpecin{\gamma_+}{\e^{-w_{N,+}}\indicator{\hat\Omega'_+}}
 \geqs 1 - \e^{-c_-/\eps}\;.
\end{equation} 
\end{prop}
\begin{proof}
Recall that
$ w_{N,+}$ has zero expectation  under $\gamma_{+}$.
Jensen's inequality yields
\begin{align*}
 \bigexpecin{\gamma_+}{\e^{-w_{N,+}}\indicator{\hat\Omega'_+}} 
 &= \fPin{\gamma_+}(\hat\Omega'_+) 
 \bigecondin{\gamma_+}{\e^{-w_{N,+}}}{\hat\Omega'_+} \\ 
 &\geqs \fPin{\gamma_+}(\hat\Omega'_+) 
 \e^{-\econdin{\gamma_+}{w_{N,+}}{\hat\Omega'_+}} \\
 &= \fPin{\gamma_+}(\hat\Omega'_+) 
 \e^{-\expecin{\gamma_+}{w_{N,+}\indicator{\hat\Omega'_+}}/\fPin{\gamma_+}
(\hat\Omega'_+)}\;.
\end{align*}
Note that the event $\hat\Omega'_+$ simply says that the first marginal of the
Gaussian distribution $\gamma_+$ is larger than a constant. Since this marginal
is a one-dimensional Gaussian distribution, centred at the positive stationary
solution, standard tail estimates show that there is a constant $c_0>0$ such
that 
\[
 \fPin{\gamma_+}(\hat\Omega'_+) \geqs 1 - \e^{-c_0/\eps}\;.
\]
Furthermore, there is a constant $K$ such that uniformly in $N$ and for
$n=1,2,3,4$
\begin{align*}
 \bigabs{\bigexpecin{\gamma_+}{U_{n,N}^+\indicator{\hat\Omega'_+}}}
 &= \bigabs{\bigexpecin{\gamma_+}{U_{n,N}^+\indicator{(\hat\Omega'_+)^c}}} 
 \leqs \bigexpecin{\gamma_+}{\abs{U_{n,N}^+}\indicator{(\hat\Omega'_+)^c}}
\\
 & \leqs \bigexpecin{\gamma_+}{(U_{n,N}^+)^2}^{1/2}
 \fPin{\gamma_+}((\hat\Omega'_+)^c)^{1/2}
 \leqs K\e^{-c_0/2\eps}\;.
\end{align*}
It thus follows that there exists a constant $c_1>0$ such that
\[
 \bigexpecin{\gamma_+}{w_{N,+}\indicator{\hat\Omega'_+}}
 \geqs - c_1 \e^{-c_0/2\eps}\;,
\]
which yields the required estimate 
$\expecin{\gamma_+}{\e^{-w_{N,+}}\indicator{\hat\Omega'_+}} \geqs 1 -
\e^{-c_-/\eps}$.
\end{proof}

Combining this result with~\eqref{eq:ZN_lb11} and Corollary~\ref{cor:cap_ub},
we finally obtain the following lower bound on the expected transition times.

\begin{prop}
\label{prop:transition_time_lb}
There exists a constant $C_-$, uniform in $N$ and $\eps$, such that 
\begin{equation}
\label{eq:transition_time_lb}
 \bigexpecin{\mu_{A,B}}{\tau_B}
 \geqs 2\pi\biggpar{ \frac{\e^{3/\abs{\lambda_0}}}{\abs{\lambda_0}(\lambda_0+3)}
 \prod_{0 < \abs{k} \leqs N}
 \biggbrak{\frac{\e^{3/\lambda_k}}{1+3/\lambda_k}}}^{1/2}
 \e^{L^2/4\eps}
 \bigbrak{1 - C_-\eps}
\end{equation} 
holds for all $N\geqs1$.
\end{prop}
\begin{proof}
Plugging~\eqref{eq:ZN_lb12} into~\eqref{eq:ZN_lb11}, using the
upper bound~\eqref{eq:cap_ub20} on the capacity and substituting
in~\eqref{eq:expec_tau1}, we obtain 
\[
 \bigexpecin{\mu_{A,B}}{\tau_B} \geqs 
 \sqrt{\frac{2\pi}{\abs{\lambda_0}\eps}} 
 \sqrt{\frac{2\pi\eps}{\lambda_0+3}}
 \prod_{0<\abs{k}\leqs N} \sqrt{\frac{\lambda_k}{\lambda_k+3}}
 \e^{L^2/4\eps} \e^{3L^2C_N/2} \bigbrak{1+\Order{\eps}}\;.
\]
Using the fact that 
\[
 \frac32 L^2 C_N = \frac32 \biggpar{\frac{1}{\abs{\lambda_0}} +
\sum_{0<\abs{k}\leqs N} \frac{1}{\lambda_k}} 
\]
yields the result.
\end{proof}


\section{Upper bound on the expected transition time}
\label{sec_ub}

In this section we derive the upper bound on the expected transition time,
i.e., we will prove the upper bound on the
partition function  \eqref{eq:ZN_estimate} (in Section~\ref{ssec_ZN_upper}) and
the lower bound on the capacity \eqref{eq:cap_estimate}  (in
Section~\ref{ssec_capacity_lower}).

Inspired by~\cite{berglund2007metastability}, we decompose the field $\phi$ into
its mean (given by the zeroth Fourier coefficient $z_0$) and its (rescaled)
transversal fluctuations, see \eqref{eq:dec02}. In fact, in Proposition~3.2
in~\cite{berglund2007metastability}  it was observed that in a similar system
the potential $V_N$ could be bounded from below by a function which only depends
on the mean and a uniformly convex function in the transversal direction. In
Lemma~\ref{lem:wo3} we obtain a similar bound in our setting (but we only state
it as a $z_0$-dependent lower bound, which is all we need). The point here is
that although this lower bound diverges logarithmically as $N \to \infty$, it
does not become worse as $\eps \to 0$ and thus permits to mimic the Nelson
argument in Proposition~\ref{prop:ZN_ub01} to obtain a bound on the integral
over the transversal directions in the partition function, which does not depend
too badly on $z_0$ and $\eps$. Once this a priori bound is established we
rewrite this transversal integral once more, this time with respect to the
Gaussian reference measure $g_{N,\perp}(z_0,y_\perp)$ (in the terminology of
Quantum Field Theory this amounts to a $z_0$-dependent  change of \lq\lq
mass\rq\rq) to obtain  a sharp upper bound on the integral over the transversal
directions (in Proposition~\ref{prop:ZN_ub02}).

The argument for the lower bound on the capacity is similar to \cite{BEGK}.
Using the characterisation \eqref{eq:cap3} of $\capacity_A(B)$ the lower bound
can obtained by solving a one-dimensional variational problem in the $z_0$
direction (see \eqref{eq:cap_lb03} below). It is here where the assumption that
the sets $A$ and $B$ are not too small enters.

Finally, in Section~\ref{ssec_laplace} it remains to treat the integral in the
$z_0$ direction. Combining the bounds of the previous two sections
one obtains an upper bound on the ratio $ \frac{\cZ_N(\eps)}{2\capacity_A(B)}$
in terms of  an integral in $z_0$ over a function which depends on $z_0$  but
not on $N$. This then permits to apply standard one-dimensional Laplace
asymptotic to conclude the argument.%


\subsection{Longitudinal-transversal decomposition of the potential}
\label{ssec_decomp}

We denote the rescaled fluctuating part of the Fourier
expansion~\eqref{eq:Fourier} by
\begin{equation}
 \label{eq:dec02}
 \hat\phi_{N,\perp}(x) 
 = \hat\phi_N(x) - \frac{z_0}{\sqrt{\eps}L}
 = \sum_{0<\abs{k}\leqs N} y_k e_k(x)\;.
\end{equation} 
Note in particular the Parseval identity  
\begin{equation}
 \label{eq:dec04} 
  \int_{\T^2}\hat\phi_{N,\perp}^2(x)\6x 
 = \sum_{0 < \abs{k} \leqs N} \abs{y_k}^2\;.
\end{equation} 
Similarly to \eqref{eq:ZN_lb04} the potential can  be written in the form 
\begin{align}
\notag
 \frac1\eps V_N(z_0,y_\perp) 
={}& \frac1\eps q(z_0) +g_{N,\perp}(z_0,y_\perp)  
+ \frac{1}{4} \int_{\T^2} \frac{6 z_0^2}{L^2} H_2(\hat\phi_{N,\perp}(x), C_N) \6
x \\
 \label{eq:dec05}
&   +  \frac{1}{4} \int_{\T^2} \Bigl(\frac{4z_0}{L}  \sqrt{\eps}
H_3(\hat\phi_{N,\perp}(x), C_N) + \eps H_4
(\hat\phi_{N,\perp}(x), C_N) \Bigr) \6 x\;,
\end{align} 
where this time
\begin{align}
\nonumber
q(z_0)
&= \frac{1}{4L^2}z_0^4 - \frac12 \abs{\lambda_0} z_0^2\;, \\
g_{N,\perp}(z_0,y_\perp) &= \frac12 \sum_{0 < \abs{k} \leqs N} \lambda_k
 \abs{y_k}^2\;.
\label{eq:dec06}
\end{align} 
Here we have used the fact that by assumption $\int_{\T^2} \hat\phi_{N,\perp}(x)
\6 x =0$, so that the corresponding term drops.
The quadratic form $g_{N,\perp}$ defines a Gaussian
measure $\gamma_0^\perp$ with normalisation
\begin{equation}
 \label{eq:dec07}
  \cN^\perp_0 = 
 \prod_{0<\abs{k}\leqs N} \sqrt{\frac{2\pi}{\lambda_k}}\;.
\end{equation} 
The associated renormalisation constant is given by 
\begin{equation}
 \label{eq:dec08}
  \bigexpecin{\gamma^\perp_0}{\hat\phi_{N,\perp}^2(x)}
 = \frac{1}{L^2} \sum_{0 < \abs{k} \leqs N} 
 \frac{1}{\lambda_k}
 =: C^\perp_N = C_N - \frac{1}{L^2}\;.
\end{equation}
As before, we are interested in the Wick powers $H_n(\hat\phi_{N,\perp},
C_N^{\perp})$ with respect to this measure and set \begin{equation}
 \label{eq:dec10} 
 U^\perp_{n,N} = \int_{\T^2} H_n(\hat\phi_{N,\perp}, C_N^{\perp})  \6x\;,
 \qquad
 U^\perp_n = \lim_{N\to\infty} U^\perp_{n,N}\;.
\end{equation} 
By construction, these random variables have (under the measure 
$\gamma^\perp_0$) zero mean and a variance bounded uniformly in $N$.
Furthermore we see that  
\begin{equation}
 \label{eq:dec11} 
 \int_{\T^2}H_3(\hat\phi_{N,\perp}(x), C_N)  \6x =
 U^\perp_{3,N}\;,
\end{equation} 
owing to the fact that $\hat\phi_{N,\perp}$ has zero mean, and  
\begin{equation}
 \label{eq:dec12} 
\int_{\T^2} H_4
(\hat\phi_{N,\perp}(x), C_N) \6 x=
 U^\perp_{4,N} - \frac{6}{L^2} U^{\perp}_{2,N} 
 + \frac{3}{L^2}\;. 
\end{equation} 
The following expression for the potential then follows immediately
from~\eqref{eq:dec05} and~\eqref{A5-1}.

\begin{prop}
\label{prop:wNperp} 
The potential can be decomposed as 
\begin{equation}
 \label{eq:dec20}
 \frac{1}{\eps} V_N(z_0,y_\perp) 
 = \frac{1}{\eps}q(z_0) + q_1(z_0,\eps) + g_{N,\perp}(y_\perp) 
 + w_{N,\perp}(z_0,y_\perp)\;,
\end{equation} 
where $q(z_0)$ and $g_{N,\perp}(y_\perp)$ are given in~\eqref{eq:dec06}, and 
\begin{align}
\nonumber
q_1(z_0,\eps)  
&= -\frac{3z_0^2}{2L^2} + \frac{3\eps}{4L^2}\;,  
\\
w_{N,\perp}(z_0,y_\perp)
&= \frac{3(z_0^2 - \eps)}{2L^2} U_{2,N}^\perp 
+ \frac{z_0}{L}\sqrt{\eps} U_{3,N}^\perp +
\frac14 \eps U_{4,N}^\perp\;.
\label{eq:dec21} 
\end{align}
\end{prop}


\subsection{Upper bound on the partition function}
\label{ssec_ZN_upper}

In order to obtain an upper bound on $\cZ_N(\eps)$, we will first perform the
integration over the fluctuating modes $y_\perp$, and then the integration over
the mean value $z_0$. The basic observation is the following rewriting of
$\cZ_N(\eps)$. 

\begin{prop}
\label{prop:ZN_ub0} 
\label{eq:prop_ZN_ub1}
The partition function is given by the integral 
\begin{equation}
\label{eq:ZN_ub01}
 \cZ_N(\eps) =
 \int_{-\infty}^\infty  \e^{-q(z_0)/\eps} 
 g(z_0,\eps) \6z_0\;,
\end{equation} 
where
\begin{equation}
\label{eq:ZN_ub02}
 g(z_0,\eps) = \e^{-q_1(z_0,\eps)}
 \prod_{0<\abs{k}\leqs N} \sqrt{\frac{2\pi\eps}{\lambda_k}} 
 \bigexpecin{\gamma^\perp_0}{\e^{-w_{N,\perp}(z_0,\cdot)}}\;.
\end{equation} 
\end{prop}

By standard, one-dimensional Laplace asymptotics, we expect the
integral~\eqref{eq:ZN_ub01} to be close to
$2\sqrt{\pi\eps}\e^{L^2/4\eps}g(L,\eps)$. 
As $g$ depends strongly on $N$, one has to use some care when performing the
Laplace asymptotics. The solution to this difficulty is to not carry out the
Laplace asymptotics directly for $\cZ_N(\eps)$, but instead for the ratio
$\cZ_N(\eps)/\capacity_A(B)$, see Section~\ref{ssec_laplace} below. Our aim is
thus to bound the expectation in~\eqref{eq:ZN_ub02}. In order to apply a Nelson
estimate, we will need a lower bound on $w_{N,\perp}(z_0,y_\perp)$. In fact, for
later use (see Proposition~\ref{prop:ZN_ub02}), we will derive a lower
bound for the slightly more general quantity
\begin{equation}
 \label{eq:ZN_ub20}
 w_{N,\perp}^{(\mu)}(z_0,y_\perp) 
 = \frac{3z_0^2}{2L^2} U_{2,N}^\perp
+ \mu \biggl( -\frac{3\eps}{2L^2} U_{2,N}^\perp
+ \frac{z_0}{L}\sqrt{\eps} U_{3,N}^\perp +
\frac14 \eps U_{4,N}^\perp\biggr)\;,
\end{equation} 
where $\mu$ is a real parameter. Note in particular that
$w_{N,\perp}^{(1)}(z_0,y_\perp)=w_{N,\perp}(z_0,y_\perp)$.
The proof of the following simple but useful lower bound is inspired by
Proposition~3.2 in~\cite{berglund2007metastability}. 

\begin{lemma}
\label{lem:wo3} 
For any $N\in\N$, $z_0\in\R$ and $\mu\in(0,\frac32)$,  
\begin{equation}
 \label{eq:woNperp_lb} 
 w_{N,\perp}^{(\mu)}(z_0,y_\perp)
 \geqs -D_N(z_0,\mu,\eps)\;,
\end{equation} 
where
\begin{equation}
 \label{eq:DN_z0_eps} 
 D_N(z_0,\mu,\eps) = 
 \frac32 z_0^2 C_N 
 + \frac34 \mu\eps C_N^2L^2 
 \biggl( \frac{3}{1-2\mu/3} - 1 \biggr)\;.
\end{equation} 
\end{lemma}
\begin{proof}
Using the definition~\eqref{eq:dec10} of Wick powers, we see that 
\begin{align}
\nonumber
 w_{N,\perp}^{(\mu)}(z_0,y_\perp)
 ={}& \frac14 \int_{\T^2} \hat\phi_{N,\perp}^2(x) 
 \biggl[ \mu\eps \hat\phi_{N,\perp}^2(x) 
 + 4\mu \frac{z_0}{L} \sqrt{\eps} \hat\phi_{N,\perp}(x)
 +6 \frac{z_0^2}{L^2} - 6\mu\eps C_N \biggr] \6x \\
 &{}- \frac32 (z_0^2-\mu\eps) C_N^\perp 
 + \frac34\mu\eps (C_N^\perp)^2L^2\;,
 \label{eq:DN_proof} 
\end{align}
where we have used the fact that $C_N^\perp + \frac{1}{L^2} = C_N$ as well
as $\int_{\T^2} \hat\phi_{N,\perp}(x) \6 x =0$. 
A completion-of-squares argument shows that the term in square brackets
in~\eqref{eq:DN_proof} is bounded below by
\[
 \mu\eps \biggl(1-\frac23\mu \biggr) \hat\phi_{N,\perp}^2(x) 
 - 6\mu\eps C_N\;.
\]
Performing a second completion of squares shows that the integral
in~\eqref{eq:DN_proof} is bounded below by 
\[
\frac{\mu\eps}{4}
 \int_{\T^2} \biggl[ \hat\phi_{N,\perp}^4(x) \Bigl(1-\frac23\mu \Bigr) 
 - 6C_N \hat\phi_{N,\perp}^2(x) \biggr] \6x 
 \geqs - \frac{9\mu\eps C_N^2L^2}{4(1-2\mu/3)}\;.
\]
The result follows, bounding $C_N^\perp$ above by $C_N$.  
\end{proof}

We are now in a position to imitate the proof of
Proposition~\ref{prop:nelson1}, to show the following upper bound.

\begin{prop}
\label{prop:ZN_ub01} 
There exist constants $M(\mu)$ and $\eps_0(\mu)$, uniform in $N$, $\eps$ and
$z_0$, such that 
\begin{equation}
 \label{eq:ZN_ub05} 
 \bigexpecin{\gamma^\perp_0}
 {\e^{-w_{N,\perp}^{(\mu)}(z_0,\cdot)}} 
 \leqs 
 M(\mu) \bigl[ 1 + \sqrt{\eps} \e^{M(\mu)z_0^2\log(1+z_0^2)/\sqrt{\eps}}
\bigr]
\end{equation} 
holds for any $\mu\in(0,\frac32)$ and all $\eps < \eps_0(\mu)$. 
\end{prop}
\begin{proof}
We will give the proof for $w_\perp^{(\mu)} =
\lim_{N\to\infty}w_{N,\perp}^{(\mu)}(z_0,\cdot)$, since the
same proof applies for any finite $N$. 
To be able to apply the integration-by-parts formula 
\begin{align*}
 \bigexpecin{\gamma^\perp_0}{\e^{-w_\perp^{(\mu)}}} 
 &\leqs \e + \int_1^\infty \e^t
 \bigprobin{\gamma^\perp_0}{-w_\perp^{(\mu)} > t}\6t \\
 &= \e \Bigbrak{1 + \int_0^\infty \e^t
\bigprobin{\gamma^\perp_0}{-w_\perp^{(\mu)} > 1+t}\6t}\;,
\end{align*}
we have to estimate $\probin{\gamma^\perp_0}{-w_\perp^{(\mu)} > t}$ when
$t\geqs 1$. For any such $t$, we pick an $N(t)\in\N$ such that 
\[
 t - D_{N(t)}(z_0,\mu,\eps) \geqs 1\;.
\]
Note that by~\eqref{eq:DN_z0_eps}, there exists a constant $M_0(\mu)$, uniform
in $N$, $\eps$ and $z_0$, such that 
\[
 D_N(z_0,\mu,\eps) \leqs M_0(\mu) \bigbrak{\eps(\log N)^2 + z_0^2\log N}\;.
\]
The condition on $N(t)$ is thus satisfied if we impose the condition
\begin{equation}
 \eps(\log N(t))^2 + z_0^2 \log N(t) \leqs \frac{t-1}{M_0(\mu)}\;. 
\label{eq:condition_N(t)} 
\end{equation}
By Lemma~\ref{lem:wo3} and the above condition, we have 
\[
 \bigprobin{\gamma_0^\perp}{-w_\perp^{(\mu)} > t} \leqs 
 \bigprobin{\gamma_0^\perp}{\abs{w_{N(t),\perp}^{(\mu)}-w_\perp^{(\mu)}} >
1}\;. 
\]
Now observe (c.f.~\eqref{eq:ZN_ub20}) that 
\[
 w_{N(t),\perp}^{(\mu)}-w_\perp^{(\mu)} = \sum_{j=2}^4 a_j(z_0,\eps)
(U_{j,N(t)}^\perp-U_j^\perp)\;,
\]
with 
\[
 a_2 = \frac3{2L^2} (z_0^2-\mu\eps)\;, 
 \qquad
 a_3 = \frac{z_0}{L}\sqrt{\eps}\;, 
 \qquad
 a_4 = \frac14\eps\;.
\]
It follows that 
\[
 \bigprobin{\gamma_0^\perp}{\abs{w_{N(t),\perp}^{(\mu)}-w_\perp^{(\mu)}} > 1} 
 \leqs \sum_{j=2}^4 P_j\;, 
 \qquad
 P_j = \biggprobin{\gamma_0^\perp}{\abs{a_j}\abs{U_{j,N(t)}^\perp-U_j^\perp} >
\frac13}\;.
\]
For any choice of $p_j(t)\in 2\N$, we have by Markov's inequality 
\[
 P_j \leqs \abs{3a_j(z_0,\eps)}^{p_j(t)} E_j\;, 
 \qquad
 E_j = \bigexpecin{\gamma_0^\perp}{\abs{U_{j,N(t)}^\perp-U_j^\perp}^{p_j(t)}}
\]
where by Nelson's estimate~\eqref{eq:A-Nelson} 
\[
 E_j {}\leqs \const (p_j(t)-1)^{jp_j(t)/2} 
 \bigexpecin{\gamma_0^\perp}{\abs{U_{j,N(t)}^\perp-U_j^\perp}^2}^{p_j(t)/2}
 {}\leqs \const \frac{(p_j(t)-1)^{jp_j(t)/2}}{N(t)^{p_j(t)}}\;.
\]
A possible choice is to take (where $a\wedge b:=\min\set{a,b}$ and $\intpart{a}$
is the integer part of $a$)
\begin{align*}
 p_j(t) &= 2 \biggintpart{(t-1)^{1/2}\wedge \frac{z_0^2}{\sqrt\eps}} \\
 \log N(t) &= \frac1{M_1}\biggbrak{\biggpar{\frac{t-1}{\eps}}^{1/2} \wedge
\frac{t-1}{z_0^2}}\;,
\end{align*}
with $M_1$ large enough to satisfy~\eqref{eq:condition_N(t)}. Indeed, this
yields $\log(N(t)^{p_j(t)}) \simeq c(t-1)/\sqrt{\eps}$ for some $c=c(\mu)>0$,
and thus 
\[
 \bigprobin{\gamma_0^\perp}{-w_\perp^{(\mu)} > t} {}\leqs \const
\e^{c'\log(1+z_0^2)z_0^2/\sqrt{\eps}}\e^{-c(t-1)/\sqrt{\eps}}
\]
for some $c'>0$, where the first exponential is due to the term
$\abs{3a_2}^{p_2(t)}$.
This shows that 
\[
 \int_0^\infty \e^t
\bigprobin{\gamma^\perp_0}{-w_\perp^{(\mu)} > 1+t}\6t 
\leqs \const \e^{c'\log(1+z_0^2)z_0^2/\sqrt{\eps}}
\Bigl( \frac{c}{\sqrt{\eps}} - 1 \Bigr)^{-1}
\]
if $\eps < c^2$. 
Substituting in the integration-by-parts formula proves the claim.
\end{proof}

Our aim is now to sharpen this bound by applying a similar trick as in the proof
of Proposition~\ref{prop:nelson2}. To this end, it will be convenient to work
with Gaussian measures $\gamma_{z_0}^\perp$, defined by the quadratic form 
\begin{equation}
 \label{eq:ZN_gamma_z0}
 g_{N,\perp,z_0}(y_\perp) 
 = \sum_{0<\abs{k}\leqs N} 
 \biggl[ \lambda_k + \frac{3z_0^2}{L^2} \biggr] \abs{y_k}^2\;.
\end{equation} 
The following result allows converting between expectations with respect to
$\gamma^\perp_0$ and $\gamma^\perp_{z_0}$.

\begin{lemma}
\label{lem:wo2} 
For any random variable $X=X(y_\perp)$ integrable with respect to
$\gamma^\perp_0$, 
\begin{equation}
 \label{eq:ZN_ub03} 
 \expecin{\gamma^\perp_0}{X} 
 = K(z_0) \expecin{\gamma^\perp_{z_0}}
 {X \e^{3z_0^2 U_{2,N}^\perp/2L^2}}\;,
\end{equation} 
where 
\begin{equation}
\nonumber
 K(z_0) = 
 \biggl[ \prod_{0 < \abs{k}\leqs N} 
 \frac{\e^{3z_0^2/L^2\lambda_k}}{1 + 3z_0^2/L^2\lambda_k}
 \biggr]^{1/2}\;.
\label{eq:ZN_ub04} 
\end{equation} 
\end{lemma}
\begin{proof}
This follows from a short computation, writing out explicitly the density of
$\gamma^\perp_0$ and expressing $\sum_k \abs{y_k}^2$ in terms of
$U_{2,N}^\perp$.
\end{proof}

\begin{remark}
\label{rem:K(z0)}
Writing $\zeta_k = 3z_0^2/L^2\lambda_k$ and using the fact that the Taylor
series of $\log(1+\zeta_k)$ is alternating, we obtain 
\begin{equation}
 \label{eq:ZN_logK}
 2\log K(z_0) = \sum_{0<\abs{k}\leqs N} 
 \bigl[ \zeta_k - \log(1+\zeta_k) \bigr] 
 \leqs \frac12 \sum_{0<\abs{k}\leqs N} \zeta_k^2 
 {}\leqs \const z_0^4\;.
\end{equation} 
This shows that $K(z_0) \leqs \e^{M_1 z_0^4}$ for some constant $M_1$,
independent of $N$ and $z_0$.
\end{remark}

We can now state the sharper bound on the expectation of $\e^{-w_{N,\perp}}$.

\begin{prop}
\label{prop:ZN_ub02}
There exists a constant $M>0$, uniform in $N$, $\eps$ and $z_0$, such that 
\begin{equation}
 \label{eq:ZN_ub06}
 \bigexpecin{\gamma^\perp_0}{\e^{-w_{N,\perp}(z_0,\cdot)}} 
 \leqs K(z_0) \Bigl[ 1 + 
 M \sqrt{\eps} (1+\abs{z_0}) 
 \bigl(1 + \sqrt\eps \e^{Mz_0^2\log(1+z_0^2)/\sqrt{\eps}}\bigr)
 \Bigr]\;.
\end{equation} 
\end{prop}
\begin{proof}
To lighten the notation, we drop the argument $(z_0,\cdot)$ of
$w_{N,\perp}$.
By Lemma~\ref{lem:wo2}, we have 
\[
 \bigexpecin{\gamma^\perp_0}{\e^{-w_{N,\perp}}} 
 = K(z_0)\bigexpecin{\gamma^\perp_{z_0}}{\e^{-\hat w_{N,\perp}}} 
\]
where
\[
 \hat w_{N,\perp} = 
 -\frac{3\eps}{2L^2} U_{2,N}^\perp
 + \frac{z_0}{L}\sqrt{\eps} U_{3,N}^\perp +
 \frac14 \eps U_{4,N}^\perp\;.
\]
As in the proof of Proposition~\ref{prop:nelson2}, we write  
\begin{align*}
\bigexpecin{\gamma^\perp_{z_0}}{\e^{-\hat w_{N,\perp}}}  
&\leqs 1 + \bigexpecin{\gamma^\perp_{z_0}}{(\e^{-\hat w_{N,\perp}} -
1)\indexfct{\hat w_{N,\perp}<0}} \\
&\leqs 1 +
\bigexpecin{\gamma^\perp_{z_0}}{\abs{\hat
w_{N,\perp}}\e^{-\hat w_{N,\perp}}\indexfct{ \hat w_{N,\perp}<0}}
\\
&\leqs 1 + \bigexpecin{\gamma^\perp_{z_0}}{(\hat
w_{N,\perp})^p}^{1/p} 
\bigexpecin{\gamma^\perp_{z_0}}{\e^{-q\hat w_{N,\perp}}}^{1/q}\;.
\end{align*}
In the last line, we have used H\"older's inequality, and $p,q\geqs1$ are
H\"older conjugates. It follows from~\eqref{A13} that 
\[
 \bigexpecin{\gamma^\perp_{z_0}}{(\hat
w_{N,\perp})^p}^{1/p} {}\leqs \const \sqrt{\eps}(1+\abs{z_0})\;.
\]
Furthermore, another application of Lemma~\ref{lem:wo2} yields
\begin{align*}
 \bigexpecin{\gamma^\perp_{z_0}}{\e^{-q\hat w_{N,\perp}}} 
 &= \frac{1}{K(z_0)} 
 \bigexpecin{\gamma^\perp_0}{\e^{-q\hat w_{N,\perp} - 3z_0^2
U_{2,N}^\perp/2L^ 2}}\\
 &= \frac{1}{K(z_0)} 
 \bigexpecin{\gamma^\perp_0}{\e^{-w^{(q)}_{N,\perp}}}\;. 
\end{align*}
Applying Proposition~\ref{prop:ZN_ub01} for some $q\in(1,\frac32)$  
and combining the different estimates yields the result. 
\end{proof}


\subsection{Lower bound on the capacity}
\label{ssec_capacity_lower}

Assume that $A=-I\times A_\perp$ and $B=I\times A_\perp$ where $\pm I =
\pm[1-\delta,1+\delta]$ (with $0<\delta<1$) are small intervals around the two
stationary solutions. Let $D=J\times D_\perp$ where $J=[-\rho,\rho]$ is an
interval joining $-I$ and $I$, with $\rho=1-\delta$, and $D_\perp \subset
A_\perp$ (see \figref{fig:lowerbound} -- when working with Fourier
variables, coordinates are scaled by a factor $L$). Then we have 
\begin{align}
\nonumber
 \capacity_A(B) &\geqs \eps\int_D \biggabs{\dpar{h_{A,B}(z)}{z_0}}^2
\e^{-V_N(z)/\eps}\6z \\
& \geqs \eps\int_{D_\perp} 
\biggbrak{\inf_{f\colon f(-\rho)=1, f(\rho)=0} \int_{-\rho}^\rho
\e^{-V_N(z_0,z_\perp)/\eps}f'(z_0)^2\6z_0} \6z_\perp\;.
\label{eq:cap_lb01} 
\end{align}
Writing the Euler--Lagrange equations, it is easy to see that the minimiser for
the term in brackets is such that 
\begin{equation}
 \label{eq:cap_lb02} 
 f'(z_0) = \frac{\e^{V_N(z_0,z_\perp)/\eps}}{
 \displaystyle \int_{-\rho}^\rho
\e^{V_N(y,z_\perp)/\eps}\6y}\;.
\end{equation} 
This yields the lower bound 
\begin{equation}
 \label{eq:cap_lb03} 
 \capacity_A(B) \geqs \eps\int_{D_\perp} \frac{1}{
 \displaystyle \int_{-\rho}^\rho
\e^{V_N(z_0,z_\perp)/\eps}\6z_0} \6z_\perp\;.
\end{equation} 

\begin{prop}
\label{prop:cap_lb01} 
There exists a constant $c_->0$, uniform in $\eps$ and $N$, such that 
\begin{equation}
 \label{eq:cap_lb20}
 \capacity_A(B) \geqs 
 \sqrt{\frac{\eps\abs{\lambda_0}}{2\pi}}
 \prod_{0<\abs{k}\leqs N}
 \sqrt{\frac{2\pi\eps}{\lambda_k}}
 \bigl[ \fP^{\gamma_0^\perp}(\hat D_\perp) - c_-\sqrt\eps \, \bigr]\;,
\end{equation} 
where $\hat D_\perp = D_\perp/\sqrt{\eps}$. 
\end{prop}
\begin{proof}
We start by obtaining a lower bound on $V_N$ in which $z_0$ is decoupled from
the transverse coordinates. Using the expression~\eqref{eq:dec20} obtained in
Proposition~\ref{prop:wNperp} and the elementary inequality
$2\abs{ab}\leqs(a^2/c+cb^2)$ for $c>0$, we obtain 
\[
 \frac{1}{\eps} V_N(z_0,y_\perp) 
 \leqs \frac{1}{\eps}q(z_0) + q_1(z_0,\eps) + \frac{z_0^2}{2L^2}
 + \frac{3z_0^4}{4L^2\sqrt{\eps}} 
 + g_{N,\perp}(y_\perp) + \sqrt{\eps}R(y_\perp,\eps)\;,
\]
where 
\[
 R(y_\perp,\eps) = 
 \frac{3}{4L^2}(U_{2,N}^\perp)^2 
 + \frac12 \sqrt{\eps}(U_{3,N}^\perp)^2 
 - \frac{3\sqrt\eps}{2L^2}U_{2,N}^\perp
 + \frac14\sqrt\eps U_{4,N}^\perp\;.
\]
Substituting in~\eqref{eq:cap_lb03} (and taking into account the scaling
$z_\perp=\sqrt{\eps} y_\perp$) yields 
\[
 \capacity_A(B) \geqs \frac{\eps}{\cJ} \prod_{0 < \abs{k} \leqs N} 
 \sqrt{\frac{2\pi\eps}{\lambda_k}}\,
 \bigexpecin{\gamma^\perp_0}{\e^{-\sqrt{\eps}R}\indicator{\hat D_\perp}}\;,
\]
where 
\[
 \cJ = \int_{-\rho}^\rho 
 \exp\biggl\{ \frac{1}{\eps}q(z_0) + q_1(z_0,\eps) + \frac{z_0^2}{2L^2}
 + \frac{3z_0^4}{4L^2\sqrt{\eps}} \biggr\} \6z_0\;.
\]
Since $q(z_0)$ has a quadratic maximum on $[-\rho,\rho]$ at $0$, standard
one-dimensional Laplace asymptotics (see for instance~\cite[Chapter~3,
Theorems~7.1 and~8.1]{Olver_book}) show that 
\[
 \cJ = \sqrt{\frac{2\pi\eps}{\abs{\lambda_0}}} 
 \bigl[ 1 + \Order{\sqrt{\eps}\,} \bigr]\;.
\]
Furthermore, Jensen's inequality implies that 
\begin{align*}
 \bigexpecin{\gamma^\perp_0}{\e^{-\sqrt{\eps}R}\indicator{\hat D_\perp}}
 &= \bigecondin{\gamma^\perp_0}{\e^{-\sqrt{\eps}R}}{\hat D_\perp} 
 \fPin{\gamma^\perp_0} (\hat D_\perp^c) \\
 &\geqs \e^{-\sqrt\eps \econdin{\gamma^\perp_0}{R}{\hat D_\perp}}
 \fPin{\gamma^\perp_0} (\hat D_\perp^c) \\
 &\geqs \biggl( 1 - \sqrt\eps\frac{\expecin{\gamma^\perp_0}{\abs{R}}}
 {\fPin{\gamma^\perp_0} (\hat D_\perp^c)}\biggr)
 \fPin{\gamma^\perp_0} (\hat D_\perp^c)\;.
\end{align*}
Since $\expecin{\gamma^\perp_0}{\abs{R}}$ is bounded uniformly
by~\eqref{A13}, the result follows.
\end{proof}

The lower bound on the capacity is thus complete, provided we take $D_\perp$
large enough to capture almost all the mass of $\gamma^\perp_0$. A possible
choice is as follows. 

\begin{lemma}
\label{lem:cap_lb02}
Assume 
\begin{equation}
 \label{eq:cap_lb11}
 D_\perp \supset \prod_{0<\abs{k}\leqs N} [-a_k,a_k]
 \qquad
 \text{with }
 a_k = \sqrt{\frac{4\eps\log(\eps^{-1})[1+\log\lambda_k]}{\lambda_k}}\;.
\end{equation} 
Then for sufficiently small $\eps$, one has 
\begin{equation}
 \label{eq:cap_lb12}
 \fPin{\gamma^\perp_0} (\hat D_\perp^c) = \Order{\eps}\;.
\end{equation} 
\end{lemma}
\begin{proof}
Standard Gaussian tail estimates show that 
\begin{align*}
\fPin{\gamma^\perp_0} (\hat D_\perp^c)
&\leqs \sum_{0<\abs{k}\leqs N} 2\e^{-a_k^2\lambda_k/2} \\
&= 2 \sum_{0<\abs{k}\leqs N} 
\exp \bigl\{ -2 \log(\eps^{-1})[1+\log\lambda_k]\bigr\} \\
&\leqs 2\eps^2 \sum_{0<\abs{k}\leqs N} \lambda_k^{-2\log(\eps^{-1})}\;.
\end{align*}
The last sum is bounded uniformly in $N$ if $\eps\leqs\e^{-1}$.
\end{proof}

Note that if $D_\perp$ is a ball in a Sobolev space as specified in
Definition~\ref{def:Dperp}, then it indeed satisfies~\eqref{eq:cap_lb11}.


\subsection{Laplace asymptotics and transition times}
\label{ssec_laplace}

Combining the results from the last two subsections, we finally obtain the
following upper bound on the expected transition time.

\begin{prop}
\label{prop:transition_time_ub}
There exists a constant $C_+$, uniform in $N$ and $\eps$, such that 
\begin{equation}
 \label{eq:laplace04}
 \bigexpecin{\mu_{A,B}}{\tau_B}
 \leqs  2\pi\biggpar{
\frac{\e^{3/\abs{\lambda_0}}}{\abs{\lambda_0}(\lambda_0+3)}
 \prod_{0 < \abs{k} \leqs N} 
 \biggbrak{\frac{\e^{3/\lambda_k}}{1+3/\lambda_k}}}^{1/2}
 \e^{L^2/4\eps}
\bigbrak{1 + C_+\sqrt\eps\,}
\end{equation} 
holds for all $N\geqs1$.
\end{prop}
\begin{proof}
If follows from Proposition~\ref{prop:ZN_ub0}, Proposition~\ref{prop:ZN_ub02}
and Proposition~\ref{prop:cap_lb01} that
\[
 \frac{\cZ_N(\eps)}{2\capacity_A(B)}
 \leqs \sqrt{\frac{2\pi}{\eps\abs{\lambda_0}}}
  \int_0^\infty  \e^{-q(z_0)/\eps} 
 \hat g(z_0,\eps) \6z_0\;,
\]
where 
\[
 \hat g(z_0,\eps)
 \leqs \e^{-q_1(z_0,\eps)} 
 K(z_0) \Bigl[ 1 + 
 M' \sqrt{\eps} (1+\abs{z_0}) 
 \bigl(1 + \sqrt\eps \e^{Mz_0^2\log(1+z_0^2)/\sqrt{\eps}}\bigr)
 \Bigr]
\]
for a constant $M'\geqs M$. In particular, we have 
\[
 \hat g(L,0) 
 \leqs \e^{3/2} K(L) 
 = \biggpar{\e^{3/\abs{\lambda_0}}
 \prod_{0 < \abs{k} \leqs N} 
 \biggbrak{\frac{\e^{3/\lambda_k}}{1+3/\lambda_k}}}^{1/2}\;.
\]
Since $q$ reaches its minimum $-L^2/4$ on $\R_+$ in $z_0=L$, 
writing 
\[
 \frac{1}{\sqrt{\eps}} \int_0^\infty  \e^{-[q(z_0)-q(L)]/\eps} 
 \hat g(z_0,\eps) \6z_0
 \leqs \cI_0 + \sqrt{\eps}\cI_1 + \eps \cI_2
\]
and applying one-dimensional Laplace asymptotics, we obtain 
\[
 \cI_0 \leqs \sqrt{\pi} \hat g(L,0) 
 \bigl( 1+C_+\sqrt{\eps} \bigr)\;,
\]
for the leading term, while $\cI_1$ and $\cI_2$ are bounded. 
\end{proof}


\appendix

\section{Hermite Polynomials and Wick Powers}
\label{sec_Wick}
In this appendix we recall some well-known facts about Hermite polynomials $H_n
= H_n(X,C)$
we use throughout the article.
Recall that  they  are defined recursively by setting 
\begin{equation}
\label{e:def-Hermite1}
\begin{cases}
H_0 = 1\;, \\
H_{n} = X H_{n-1} - C \, \partial_X H_{n-1} & \qquad  n \in \N\;.
\end{cases}
\end{equation}
In particular, we have 
\begin{align}
\notag
H_1(X,C) &= X & H_2(X,C) &= X^2 - C \\
H_3(X,C) &= X^3 - 3 CX & H_4(X,C) &= X^4 - 6 CX^2 +3C^2\;.
\label{eq:A2} 
\end{align}
The following binomial identity for Hermite polynomials is well-known.
\def \daPrato {\cite[Lem. 3.1]{daPratoDebussche}}
\begin{lemma}[Binomial formula for Hermite polynomials, \daPrato]
We have for any $n \in \N$ and $X,v,C \in \R$ 
\begin{equation}
\label{A3-1} 
H_n(X+v,C) = \sum_{k=0}^n \binom{n}{k} \, H_{n-k}(X,C)  \,v^k\;.
\end{equation}
\end{lemma}

We will mostly be interested in Hermite polynomials of centered Gaussian random
variables $X$ and  we will typically choose $C= \E [X^2]$.  In this case the
random variable $H_n(X,\E[X^2])$ is sometimes referred to as the $n$-th Wick
power of $X$ and denoted by $\Wick{X^n}$. The following identity is one of the
key properties of Wick powers.
\def \Nualart {\cite[Lemma 1.1.1]{nualart2006malliavin}}
\begin{lemma}[\Nualart]
Let $X,Y$ be centered jointly Gaussian random variables. Then 
\begin{equation}\label{e:WICK}
\E \big[(\Wick{X^n}) (\Wick{Y^m}) \big]= 
\begin{cases}
\E[ X Y]^n  & \text{if } n=m\\
0 & \text { else}\;.
\end{cases}
\end{equation}
\end{lemma}
Note that this implies in particular, that for $n \geqs 1$ we have $\E
[\Wick{X^n} ] = 0$. In some computations it is convenient for us to
change the value of the constant $C$ appearing in $H_n$. This will be
relevant when changing the Gaussian reference measure. The following
transformation rule, valid for any real $X, C, \bar{C}$ is  easy to check:
\begin{align}
\notag
H_1(X,C) &= H_1(X,\bar{C})\\
\notag
 H_2(X,C) &= H_2(X,\bar{C}) - (C-\bar{C}) \\
 \notag
H_3(X,C) &= H_3(X, \bar{C}) - 3(C-\bar{C})H_1(X, \bar{C}) \\
 H_4(X,C) &= H_4(X,\bar{C}) - 6 (C-\bar{C})H_2(X, \bar{C}) +3(C - \bar{C})^2\;.
\label{A5-1} 
\end{align}
We now use \eqref{e:WICK} to derive some classical facts about (spectral
Galerkin approximations of) the two-dimensional massive Gaussian free field and
its Wick powers. For any $N$ and for $x \in \T^2 = \R^2/ (L\Z)^2$ we consider
the random field 
\begin{equation}
\phi_N(x) = \sum_{|k| \leqs N} \frac{z_k}{\sqrt{\abs{\lambda_k + m^2}}}
e_k(x)\;,
\end{equation}
where the Fourier basis functions $e_k(x)$ and the eigenvalues $\lambda_k$ of $(-\Delta +1$) are 
 defined as in Section~\ref{sec_results}.
The $z_k$ are complex-valued Gaussian random variables  which are independent 
up to the constraint $z_k = \overline{z_{-k}}$ which makes $\phi_N$ a
real-valued field and which satisfy 
\begin{equation}
\E [z_k z_{-\ell}] = 
\begin{cases}
1  & \text{if } k=\ell\\
0 & \text { else}\;.
\end{cases}
\end{equation}
Note that due to the constraint $z_k = \overline{z_{-k}}$  these $(2N+1)^2$
dependent complex-valued Gaussian random variables can be represented in terms
of $(2N+1)^2$ independent real-valued random variables. The mass $m^2 \geqs 0$
is a parameter of the model, which in our case only takes either the value $0$
or the value $3$. 
 
For fixed $x$ we get
\begin{equation}
\E [\phi_N(x)^2] = \sum_{|k| \leqs N}  \frac{1}{\abs{\lambda_k + m^2}} 
=: C_N\;.
\end{equation} 
Note  that $C_N$ diverges  logarithmically, which suggests that the random
variables $\phi_N(x)$ for a fixed $x$ do not converge to a meaningful limit as
$N$ goes to $\infty$. However, it is well-known that for any test-function
$\psi$ the random variables $\int \phi_N(x) \psi(x) \6x$ converge in $L^2$
(with respect to probability) to a Gaussian limiting random
variable. 
We will not make use of this general fact, but only use that the integrals of
$\phi_N(x)$ as well as its Wick powers $\Wick{\phi_N^n(x)} = H_n(\phi_N(x),
C_N)$ have a uniformly-in-$N$ bounded variance. To see this we write for 
$M>N$
\begin{align}
\notag
&\E \Big[  \int_{\T^2} \Wick{\phi_M^n(x)} \6 x   \int_{\T^2} \Wick{\phi_N^n(y)} 
\6 y   \Big] \\
\notag
&= \int_{\T^2} \!\! \int_{\T^2} \E \big[ (\Wick{\phi_M^n(x)}) 
(\Wick{\phi_N^n(y)}) \big] \6 x
\6y
=  \int_{\T^2} \!\! \int_{\T^2} \E \big[\phi_M(x) \phi_N(y)\big]^n \6 x \6y\\
\notag
& =   \int_{\T^2} \!\! \int_{\T^2} \Big( \frac{1}{L}\sum_{|k| \leqs N} \frac{
e_{k}(x-y)}{\abs{\lambda_k + m^2}}
\Big)^n \6x \6y 
 = L^{2-n} \int_{\T^2} \Big( \sum_{|k| \leqs N} \frac{ e_{k}(x)}{\abs{\lambda_k
+ m^2}} \Big)^n
\6x  \\
\label{A3}
& = L^{2-2n} \sum_{\substack{k_1 + k_2 +  \dots + k_n = 0\\|k_i| \leqs N }}
\frac{1}{\abs{\lambda_{k_1} + m^2}} \dots   \frac{1}{\abs{\lambda_{k_n} +
m^2}}\;.
\end{align}
This calculation has several immediate consequences. First of all we can
conclude that as announced above the variances of $\int_{\T^2}
\Wick{\phi_N^n(x)} \6x$ 
are uniformly bounded as $N$ goes to $\infty$:  
\begin{equation}
\label{A4}
\sup_{N}  \E \big[ \big(  \int_{\T^2} \Wick{\phi^n_N(x)} \6 x \big)^2 \big] =  
L^{2-2n}\sum_{\substack{k_1 + k_2 +  \dots + k_n = 0\\k_i \in \Z^2 }}
\frac{1}{\abs{\lambda_{k_1} + m^2}} \dots   \frac{1}{\abs{\lambda_{k_n} + m^2}}
< \infty\;.
\end{equation}
Indeed, the convergence of this sum can be checked  
easily (e.g.~as in \cite[Lem. 3.10]{zhu2015three}). 
Furthermore, we get for $M >N$  
\begin{align}
\notag
&\E \big[ \big(  \int_{\T^2} \Wick{\phi_M^n(x)}  \6x -   \int_{\T^2}
\Wick{\phi_N^n(x)} \6 x  \big)^2 \big] \\
\notag
&=  L^{4-2n} \sum_{\substack{k_1 + k_2 +  \dots + k_n = 0\\|k_i| \leqs M }}
\frac{1}{\abs{\lambda_{k_1} + m^2}} \ldots   \frac{1}{\abs{\lambda_{k_n} +
m^2}} - L^{4-2n} \sum_{\substack{k_1 + k_2 +  \dots + k_n = 0\\|k_i| \leqs N }}
\frac{1}{\abs{\lambda_{k_1} + m^2}} \dots \frac{1}{\abs{\lambda_{k_n} + m^2}} \\
\label{A5}
& \leqs C_{n,L}\frac{(\log N)^{n-2}}{N^{2}}\;,
\end{align}
for a constant $C_{n,L}$  which depends on $n,L$ but not on $N,M$. 

Finally, we recall the definition of Wiener chaos which in this finite
dimensional context is the following: 
\begin{definition}
For $n \in \N_0$ the $n$-th (inhomogeneous) Wiener chaos generated by the random
variables $(z_k)_{|k|\leqs N}$ is the vector space of real-valued random
variables $X$ which can be written as polynomials of degree at most $n$ in the
finitely many random variables $z_k$.
\end{definition}
As stated this definition depends on the number of independent Gaussians used to
define the Wiener chaos. However, the following classical 
and important estimate holds true uniformly in that number. See e.g.
\cite[Thm.~4.1]{da2007wick} for a direct proof. This Theorem can also be 
deduced immediately from the hyper-contractivity of the Ornstein--Uhlenbeck
semigroup \cite[Thm.~1.4.1]{nualart2006malliavin}.
\begin{lemma}[Equivalence of moments]
\label{le:Nelson}
Let $X$ be a random variable, belonging to the $n$-th inhomogeneous Wiener
chaos. Then for any
$p \geqs 1$ one has
\begin{equation}
\label{eq:A-Nelson} 
 \E \big[ X^{2p}\big]^{\frac{1}{2p}} \leqs C_n
(2p-1)^{\frac{n}{2}} \E \big[ X^2
\big]^{\frac12}
\end{equation}
where $C_n$ only depends on $n$. 
\end{lemma}
\begin{remark}
The $n$-th \emph{homogeneous} Wiener chaos is defined as the 
orthogonal complement (with respect to the $L^2$ scalar product) 
of the $n-1$-st  inhomogeneous Wiener chaos in the $n$-th inhomogeneous 
Wiener chaos. If in the previous Lemma, $X$ takes values in the homogeneous 
Wiener chaos, then the estimate holds true with constant $C_n=1$.
\end{remark}
Now combining Lemma~\ref{le:Nelson} with \eqref{A4} we obtain for $p
\geqs1$ that
\begin{equation}
\label{A13}
\sup_{N}  \E \big[ \big(  \int_{\T^2} \Wick{\phi^n_N(x)} \6 x \big)^{2p}
\big] < \infty
\end{equation}
and combining Lemma~\ref{le:Nelson} with 
\eqref{A5} we see that for $M>N$ and any $p \geqs 1$
\begin{equation}
\label{A14}
\E \big[ \big(  \int_{\T^2} \Wick{\phi_M^n(x)} \6 x -   \int_{\T^2}
\Wick{\phi_N^n(x)}  \6x  \big)^{2p} \big]^{\frac{1}{p}} \leqs C_{n,L}
 (2p-1) \frac{(\log N)^{n-2}}{N} 
\end{equation}
for a constant $C_{n,L}$ which depends on $n,L$ but  not  on $p,M,N$.

\small
\bibliography{BGW}
\bibliographystyle{abbrv}               


\goodbreak

\vfill

\bigskip\bigskip\noindent
{\small 
Nils Berglund \\ 
Universit\'e d'Orl\'eans, Laboratoire MAPMO \\
{\sc CNRS, UMR 7349} and 
F\'ed\'eration Denis Poisson, FR 2964 \\
B\^atiment de Math\'ematiques, B.P. 6759\\
F-45067~Orl\'eans Cedex 2, France \\
{\it E-mail address: }{\tt nils.berglund@univ-orleans.fr}

\bigskip\noindent
Giacomo Di Ges\`u \\
CERMICS -- Ecole des Ponts ParisTech \\
6-8 avenue Blaise Pascal, \\
F-77455~Marne-la-Vall\'ee Cedex 2, France 

\smallskip\noindent
Current address: \\
TU Vienna, \\
E101 - Institut f\"ur Analysis und Scientific Computing, \\
Wiedner Hauptstr.\ 8, \\
1040 Vienna, Austria \\
{\it E-mail address: }{\tt giacomo.di.gesu@tuwien.ac.at}

\bigskip\noindent
Hendrik Weber \\
Mathematics Institute \\ 
University of Warwick \\
Coventry CV4 7AL, United Kingdom \\
{\it E-mail address: }{\tt hendrik.weber@warwick.ac.uk}

}


\end{document}